\documentclass[10pt,a4paper,final,oneside,openright,reqno]{amsart} 
\usepackage{latexsym, amsfonts, amsthm, amsmath, amssymb, tikz, enumitem, booktabs, float,color,multicol,epsf, url,epsfig,epstopdf, array, setspace, graphicx, csquotes, xcolor,tikz-cd}
\usepackage[top=30mm, bottom=35mm, left=30mm , right=30mm]{geometry}
\usepackage{pinlabel}
\usepackage{amsfonts}
\usepackage{amsmath}
\usepackage{amssymb}
\usepackage[applemac]{inputenc}
\usepackage{url}
\usepackage{hyperref}
\usepackage{color}
\usepackage{lipsum}
\usepackage{mathtools}
\usepackage[T1]{fontenc}
\usepackage{tikz-cd}
\usepackage{url} 
\usepackage{hyperref} 
\usepackage{pdfpages}
\usepackage{graphicx}
\usepackage{ulem}

\newtheorem{theorem}{{Theorem}}[section]

\newtheorem{proposition}[theorem]{{Proposition}}
\newtheorem{definition}[theorem]{{Definition}}
\newtheorem{lemma}[theorem]{{Lemma}}
\newtheorem{corollary}[theorem]{{Corollary}}
\newtheorem{fact}[theorem]{{Fact}}

\newtheorem{remark}[theorem]{{Remark}}

\newtheorem{question}[theorem]{{Question}}
\newtheorem{observation}[theorem]{{Observation}}

\newtheorem{comment}[theorem]{{Comment}}
\newtheorem{problem}[theorem]{{Problem}}
\newtheorem{example}[theorem]{{Example}}

\definecolor{greenbf}{rgb}{0.2, 0.8 ,0.4}
\hypersetup{colorlinks=true, breaklinks=true, urlcolor= blue, linkcolor= blue, citecolor= {blue},}

\def\R{\mathbb{R}}

\def\H{\mathbb{H}}
\def\Z{\mathbb{Z}}
\def\S{\mathbb{S}}

\def\T{\mathbb{T}}

\def\F{\mathcal{F}}

\def\Isom{{\sf{Isom}}}

\def\Aff{\sf{Aff}}

\def\SOL{{\sf{SOL}}}

\def\SOL{{\mathsf{Sol}} }

\def\heis{{\mathfrak{heis}} }
\def\sol{{\mathfrak{sol}} }

\def\aff{{\mathfrak{aff}} }

\def\Heis{{\sf{Heis}}}

\def\PSL{{\sf{PSL}}}

\def\O{{\sf{O}}}

\def\Span{{\sf{span}} }

\def\Mink{{\sf{Mink}} }

\def\AdS{{\sf{AdS}} }

\def\g{{\mathfrak{g}}}

\def\so{{\mathfrak{so}}}
\def\sl{{\mathfrak{sl}}}

\def\z{{\mathfrak{z}}}
\def\z{{\mathfrak{z}}}

\newcommand{\ad}{\mathrm{ad}}

\def\Heis{{\mathsf{Heis}}}
\def\Poi{{\mathsf{Poi}}}
\def\Sol{{\mathsf{Sol}}}
\def\L{{\sf{L}}}
\def\u{{\sf{u}}}
\def\G{{\widehat{G}}}
\def\Dev{{\mathsf{Dev}}}
\def\Osc{{\mathsf{Osc}}}
\def\osc{{\mathfrak{osc}}}

\begin{document}
\title[]{On completeness of Foliated structures, \\and Null Killing fields}

\author [M. Hanounah]{Malek Hanounah}

\author [L. Mehidi]{Lilia Mehidi}

\date{\today}
\maketitle
\begin{center}
\noindent Malek Hanounah\\
Institut f\"ur Mathematik und Informatik \\
Greifswald University, Germany\\
Email: malek.hanounah@uni-greifswald.de
\vspace{0.3cm}

\noindent Lilia Mehidi\\
Departamento de Geometria y Topologia\\
Facultad de Ciencias, Universidad de Granada, Spain\\
Email: lilia.mehidi@ugr.es
\end{center}

\begin{center}
\begin{abstract}
We consider a compact manifold $(M,\mathfrak{F})$ with a foliation $\mathfrak{F}$, and a smooth affine connection $\nabla$ on the tangent bundle of the foliation $T\mathfrak{F}$. 
We introduce and study a foliated completeness problem. Namely,  under which conditions on $\nabla$ the leaves are complete? We consider different natural geometric settings: the first one is the case of a totally geodesic lightlike foliation of a compact Lorentzian manifold, and the second one is the case where the leaves have particular affine structures. 
In the first case, we characterize the completeness, and obtain in particular that if a compact Lorentzian manifold admits a null Killing field $V$ such that the distribution orthogonal to $V$ is integrable, then it defines a (totally geodesic) foliation with complete leaves.
In the second case, we give a completeness result for a specific affine structure called ``the unimodular affine lightlike geometry'', and characterize the completeness for a natural relaxation of the geometry. 
On the other hand, we study the global completeness of a compact Lorentzian manifold in the presence of a null Killing field. We give two non-complete examples, starting from dimension $3$: one is a locally homogeneous manifold, and the other is a $3$D example where the Killing field dynamics is equicontinuous.
\medskip

\noindent\textbf{Mathematics Subject Classification:} 53C12, 53C50, 53C22. 
\end{abstract}
\end{center}

\addtocontents{toc}{\protect\setcounter{tocdepth}{0}}
\section*{Statements and Declarations}
\paragraph{\textbf{Competing interests:}} The authors have no competing interests to declare that are relevant to the content of this article. \\\textbf{Data availibility} Data sharing is not applicable to this article as no new data were created or analyzed in this study. \tableofcontents
\addtocontents{toc}{\protect\setcounter{tocdepth}{1}}

\section{Introduction}
Let $(M, \nabla)$ be a (connected) compact smooth manifold equipped with a smooth affine connection on the tangent bundle $TM$. A natural question is: under what conditions on the pair $(M, \nabla)$ is the manifold geodesically complete, meaning every geodesic is defined on $\mathbb{R}$? When $\nabla$ is the Levi-Civita connection of a smooth Riemannian metric $g$ on $M$, such a couple is geodesically complete. However, for non-Riemannian metrics, for instance in Lorentzian signature, this statement does not hold in general, making the study more rich. This leads to exploring sufficient conditions, such as curvature constraints, or the existence of Killing (or parallel) vector fields, for a compact Lorentzian manifold to be geodesically complete \cite{carriere1989autour, romero1995completeness, klingler1996completude, leistner2016completeness, mehidi2022completeness}.

\subsection{Leafwise completeness.} 
In various settings, one may naturally ask if some classical results extend to the context of foliations, i.e. whether they have a `foliated' counterpart. A well-known example  is the uniformization problem of $2$-dimensional foliations \cite{ghys_uniformisation} (or more generally, laminations), which asks whether a Riemannian metric
$g$ defined on the tangent bundle of a codimension $1$ foliation (or lamination) of a compact manifold of dimension $3$, can be made conformal to a metric $h$ with constant curvature along the leaves. Another example  asks whether leafwise harmonic functions for foliations on compact manifolds have the same behavior as harmonic functions on compact manifolds \cite{feres2005dynamics}.
In this paper, we introduce foliated versions of  completeness problems. Namely, we consider a triple $(M,\mathfrak{F}, \nabla_\mathfrak{F})$, where $M$ is a compact manifold, $\mathfrak{F}$ a foliation of $M$, and $\nabla_{\mathfrak{F}}$ a smooth affine connection defined on $T\mathfrak{F}$. The completeness problem above can be seen as the codimension $0$ case in this setting. We formulate the problem as follows 
\begin{problem}\label{Problem: FC}
    Let $(M,\mathfrak{F}, \nabla_\mathfrak{F})$ be a triple as above. Are there natural (geometric) settings in which the leaves of $\mathfrak{F}$ are geodesically complete?
\end{problem}

Of course, when the connection on the bundle $T \mathfrak{F}$ is induced by the Levi-Civita connection of a Riemannian metric on the manifold, the leaves are complete. We refer to this as a `Riemannian situation': this is an elementary situation from the point of view of the completeness problem. Therefore, we are interested in non-Riemannian situations. In the following question, we consider  a natural and interesting non-Riemannian situation: 
\begin{question}[\textbf{Foliated lightlike geometry: leafwise completeness}]\label{Question: FC with degenerate Riem metric}
 Let $(M,\mathfrak{F}, \nabla_\mathfrak{F})$ be a triple as defined above. Assume that we have a degenerate Riemannian metric on the tangent bundle $T\mathfrak{F}$, such that on each leaf, the induced degenerate Riemannian metric is parallel with respect to the given connection. Assume further that the radical (kernel) of the metric has dimension $1$, and denote by $\mathcal{N}$ the $1$-dimensional foliation it defines. Observe that the leaves of $\mathcal{N}$ are totally geodesic (indeed, the degenerate Riemannian metric on each leaf is parallel, so its radical is invariant by parallel transport).
 Is it true that the leaves are complete if and only if the geodesics tangent to $\mathcal{N}$ are complete?
\end{question}
We do not give an answer to Question \ref{Question: FC with degenerate Riem metric} in  full generality, but this question contains many natural situations, which we study in this paper. One of them is that of a totally geodesic lightlike (i.e. the metric induced on each leaf is Riemannian degenerate) foliation $\mathfrak{F}$  of a compact Lorentzian manifold. In this case, the connection on the tangent bundle $T\mathfrak{F}$ is the restriction to $T\mathfrak{F}$ of the Levi-Civita connection of the Lorentzian metric.  
Another natural and important situation is when $\nabla$ is a flat connection. In this case, the leaves have some specific affine structure. 
The latter case actually belongs to a more general situation arising from Problem \ref{Problem: FC}: that of a foliation with a $(G,X)$-structure on the leaves. We introduce the notion of a $(G,X)$-foliation in Definition \ref{Defi: GX foliation}, and we focus on certain affine structures with interesting application in the study of completeness in Lorentzian (or semi-Riemannian) geometry. In this case, the geodesic completeness is equivalent to the completeness of the $(G,X)$-structure when the model is complete. We refer to Subsection \ref{Intro-Subsection: (G,X)-foliated} of the introduction for the treatment of the problem in this setting. \\

To summarize, we consider the two following situations  
\begin{itemize}
    \item[\textbf{(A)}] A compact Lorentzian manifold with a totally geodesic  lightlike  foliation. This solves Question \ref{Question: FC with degenerate Riem metric} when the connection is induced by an ambient Lorentzian Levi-Civita connection. 
    \item[\textbf{(B)}]  A foliated compact manifold where the leaves have an affine structure  in the sense of $(G,X)$-structures (see Definition \ref{Defi: GX foliation}). Observe that here, there is no assumption on existence of an ambient connection. The assumptions in Question \ref{Question: FC with degenerate Riem metric}, combined with the situation in Item \textbf{(B)}, give rise to the so-called `affine lightlike geometry' (see Subsection \ref{Intro-Subsection: (G,X)-foliated}) on the leaves. We characterize the completeness in this case, providing an answer to Question \ref{Question: FC with degenerate Riem metric} when $\nabla$ is a flat connection. Finally, we present other natural affine structures where geodesic completeness may fail. 
\end{itemize}
For Item \textbf{(A)}, we prove the following
\begin{theorem}\label{Introduction: theorem tot geod lightlike foliation completeness}
    Let $M$ be a  compact Lorentzian manifold. Let $\mathfrak{F}$ be a codimension $1$ totally geodesic lightlike foliation of $M$. Then the leaves of $\mathfrak{F}$ are geodesically complete with respect to the induced connection if and only if the null geodesics in the leaves of $\mathfrak{F}$ are complete.
\end{theorem}
Theorem \ref{Introduction: theorem tot geod lightlike foliation completeness} includes the class of (compact) Brinkmann spacetimes. A Brinkmann spacetime is a Lorentzian manifold with a parallel null vector field $V$, so the codimension $1$ foliation $\mathfrak{F}$ tangent to $V^\perp$ is lightlike and totally geodesic. It is proved in \cite{mehidi2022completeness} that a compact Brinkmann spacetime is complete. In particular, the completeness of geodesics tangent to $\mathfrak{F}$ is a leafwise completeness result. Theorem \ref{Introduction: theorem tot geod lightlike foliation completeness} can be viewed as a generalization of the leafwise completeness to the situation where only a totally geodesic foliation is given. This includes, for instance, the broader class of compact (locally) \textit{Kundt spacetimes} (see \cite{mzkundt} for a recent introductory survey).
\medskip

As an application, we get the following consequences:
\begin{corollary}\label{Intro-Cor: parallel null line complete}
     Let $(M,l)$ be a compact Lorentzian manifold with a parallel null line field $l$.  The foliation tangent to $l^\perp$ is geodesically complete if and only if the (null) geodesics tangent to $l$ are complete. 
\end{corollary}
In \cite[Section 4.3, Example 3]{leistner2016completeness}, the authors ask whether a compact Lorentzian manifold with a parallel null line field $l$ has a foliation (tangent to $l^\perp$) with complete leaves. The corollary above gives a necessary and sufficient condition for this to be true. In addition, we give an example where this condition is not satisfied, resulting in incomplete leaves.  
This example is constructed as the product of a circle with a Clifton-Pohl torus (Example \ref{Example: CP higher dimension}).
\begin{corollary}\label{Intro-Cor: Killing-complete}
     Let $(M,V)$ be a compact Lorentzian manifold with a null Killing field. Assume that the distribution $V^\perp$ is integrable. Then, the foliation tangent to $V^\perp$ is totally geodesic, and a geodesic tangent to it is complete. 
\end{corollary}

\subsection{On the geodesic completeness in presence of a null Killing field} Given a compact Lorentzian manifold with a null Killing field, one could extend the question to global completeness, that is, the completeness of the entire manifold $M$. One motivation for this is the completeness result obtained in \cite{mehidi2022completeness} when $V$ is additionally parallel. 

When the metric admits a timelike Killing field $V$, it was observed in \cite{romero1995completeness} that the intersection of the levels of the metric with the levels of $g(\,\cdot,V)$ are compact, which leads to completeness. However, this is no longer the case when $V$ is spacelike or lightlike. 
In fact, in the spacelike case, taking a product of a non-complete Lorentzian manifold with the circle (endowed with a Riemannian metric) shows that existence of a spacelike Killing field does not imply completeness. But there is no known incomplete example in the lightlike case. 
\subsubsection{\textbf{An incomplete example}} We construct in this paper two examples of incomplete compact Lorentzian manifolds with a null Killing field. The first one is locally homogeneous. We state

\begin{theorem}
    There exists an incomplete left invariant Lorentzian metric on $G:=\Sol \times\R$ with a central null Killing field $V$. Taking a quotient $M:=\Gamma\backslash G$, where $\Gamma$ is a cocompact lattice of $G$, defines an incomplete compact Lorentzian manifold  on which the flow of $V$ acts isometrically.
\end{theorem}
\subsubsection{\textbf{The case of dimension $3$}} It is a simple fact that a compact Lorentzian surface with a null Killing field is complete (in fact, it is even flat). Our second incomplete example is more elaborate and is constructed in dimension $3$. 
Let $(M,V)$ be a compact Lorentzian manifold of dimension $3$ with $V$ a null Killing field. The first observation is that, in this dimension,  the distribution $V^\perp$ is integrable and totally geodesic, and 
the geodesics tangent to $V^\perp$ are complete by Corollary \ref{Intro-Cor: Killing-complete}.   However, for the question of global completeness, the situation is different depending on the dynamics of $V$: when $V$ has complicated (non-equicontinuous) dynamics, the manifold is complete; in the opposite (equicontinuous) case, we construct an incomplete example (see details below). \medskip

\paragraph{\textbf{I) Non-equicontinuous case}} This means that the closure of the flow of $V$ in the isometry group is non-compact. The global behavior of $V$ is well understood \cite{zeghib1996killing}: Zeghib classifies all such flows on Lorentzian $3$-manifolds, and it follows from the classification that $M$ is complete  \cite[Remark 16.1]{zeghib1996killing}.  The case where $V$ is non-equicontinuous leads to a non-compact isometry group. In this more general case, namely, when the isometry group of a compact Lorentzian $3$-manifold is non-compact, the manifold is complete by Frances's classification results \cite{frances}. The question remains open in higher dimensions, and completeness of geodesics transversal to $V^\perp$ seems non-trivial in general.\medskip   

\paragraph{\textbf{II) Equicontinuous case}}
The case where the flow of $V$ is equicontinuous remains widely open. In this case, the closure of the flow of $V$ in the isometry group is an abelian compact subgroup, hence a torus of some dimension $k$. We show that  if $k\geq3$, the manifold is just a flat $3$-torus (Proposition \ref{Prop: structure of G}). When $k=2$, we show completeness in some specific situations, namely, when the $\T^2$-action is free and the metrics induced on the fibers of the $\T^2$-principal bundle have the same signature. Finally, we consider the case where $k=1$. Here,  the action of $V$ is periodic, and when it is free,  $M$ is an $\mathbb{S}^1$-bundle over a torus. Even if the dynamics of $V$ is trivial, the reduction of the geodesic equation to dimension $2$, due to the existence of a Noether first integral, leads to a complicated  situation. 
The study of this situation raises interesting questions on the relationship between the geodesic completeness and the dynamics of certain vector fields defined on the torus. We prove the following
 
\begin{theorem}\label{Intro: Thm incomplete 3d}
    There exists an incomplete compact $3$-dimensional Lorentzian manifold with a periodic null Killing field.
\end{theorem}
\bigskip

The next subsection is devoted to Part \textbf{(B)}.

\subsection{$(G,X)$-foliated completeness}\label{Intro-Subsection: (G,X)-foliated} Completeness of compact $(G,X)$-manifolds is a very rich and active topic of old and current research. It has developed in many directions, especially in the case of affine structures (see, for instance, \cite{carriere1989autour,fried1981affine, fried1986distality}). Completeness of compact $(G,X)$-manifolds is in general far from being easy. 
Here, we consider a compact manifold $M$ with a foliation $\mathfrak{F}$ having a \textbf{tangential} $(G,X)$-structure, meaning that the $(G,X)$-structure is supported on the leaves of $\mathfrak{F}$ and varies continuously from one leaf to another. We call the pair $(M,\mathfrak{F})$ a \textbf{$(G,X)$-foliated manifold}. A precise definition is given in Definition \ref{Defi: GX foliation}. 
\medskip

Note that $(G,X)$-foliated manifolds arise naturally in many interesting geometric contexts. For example,  the so-called pp-waves (Example \ref{Intro-Example: pp-waves}), and more generally, weakly pp-waves (Example \ref{Intro-Example: weakly pp-wave}), are codimension $1$ foliated Lorentzian manifolds with a rigid structure on the leaves. For a pp-wave (resp. weakly pp-wave), this structure is a  particular affine structure called the `affine unimodular lightlike geometry' (resp. the `affine lightlike geometry'). Those structures are introduced in Subsections \ref{Intro-Subsection: affine unimodular lightlike geometry} and \ref{Intro-Subsection: Affine lightlike geometry} below. 

Note that $(G,X)$-foliations are briefly mentioned in Thurston's book \cite[p.160]{thurston2014three}, but not precisely defined. However, there is some old literature dealing with tangential structures on foliations. For instance, \cite{furness1974affine, inaba1993tangentially} study foliations with an affine tangential structure. 
While these works address different questions (in \cite{furness1974affine} the author assumes completeness of the leaves to prove topological results on the manifold, while \cite{inaba1993tangentially} deals with the existence of such foliations on the $2$-torus and the $3$-sphere, independently of completeness), the present paper introduces a new perspective by focusing on the completeness problem. A more precise formulation of Item \textbf{(B)} is the following
\begin{problem}\label{Prob: foliated problem 1}
 Consider a $(G,X)$-structure for which any compact $(G,X)$-manifold is complete. Is it true that for any compact $(G,X)$-foliated manifold $(M, \mathfrak{F})$, the leaves of $\mathfrak{F}$ are $(G,X)$-complete? We call this the \textbf{$(G,X)$-foliated completeness problem}.   
\end{problem}
To our knowledge, the \textbf{$(G,X)$-foliated completeness problem} has not been considered before. Besides the fact that it is itself interesting and natural, there are some particular geometric structures for which the full completeness reduces essentially to a $(G,X)$-foliated completeness problem. 

\subsubsection{\textbf{Affine unimodular lightlike geometry}}\label{Intro-Subsection: affine unimodular lightlike geometry} 
An $(\Aff(\R^{n+1}), \R^{n+1})$-structure on a manifold $M$ is called an affine structure. For a complete survey on affine structures, one can see  \cite[Ch. 6-10]{goldman2022geometric}. We will be interested in a subgeometry of the affine geometry, which we call \textit{the affine unimodular lightlike geometry}. It was first introduced in \cite{article1}.  The structure group is a subgroup of the affine  group $\Aff(\R^{n+1})$, given by $$\L_\u(n,0):=\left\{\begin{pmatrix}
1 & \alpha^{\top} \\
0 & A 
\end{pmatrix}\ltimes \R^{n+1} \ | \  \alpha\in \R^{n}, \  A\in \O(n)\right\},$$
where $\mathsf{O}(n)$ is the orthogonal group of the Euclidean scalar product on $\R^n$. The $(\mathsf{L}_\u(n,0), \R^{n+1})$-geometry is the maximal affine geometry with two $\mathsf{L}_\u(n,0)$-invariant objects: a constant vector field $V :=\partial_{x_1}$ and a degenerate Riemannian metric $h:=dx_2^2+dx_3^2+ \ldots +dx_{n+1}^2$ with radical $\R V$. The notation $(n,0)$ refers to the signature of this degenerate Riemannian metric on $\R^{n+1}$, and is useful to distinguish between different signatures in the next sections. When the second entry is $0$, we simply write $\L_\u(n)$ instead of $\mathsf{L}_\u(n,0)$.

\begin{example}[pp-waves are $(\mathsf{L}_\u(n), \R^{n+1})$-foliated manifolds]\label{Intro-Example: pp-waves}
    A pp-wave  is a Lorentzian manifold $(M,V)$ with a parallel null vector field $V$, such that the (integrable) parallel distribution $V^\perp$ is tangent to a (totally geodesic) foliation $\mathfrak{F}$ with flat leaves (flat with respect to the induced connection). It admits ``Brinkmann local coordinates'' of the form
    \begin{align*}
     2 du dv + H(u,x) du^2 + \sum_{i=1}^n (dx^i)^2.
    \end{align*} 
    A pp-wave of dimension $n+2$, with its foliation $\mathfrak{F}$, provides a natural example of a $(\mathsf{L}_\u(n), \R^{n+1})$-foliated manifold.
    Indeed, it is well known that the data of a flat affine connection on a smooth manifold of dimension $n+1$ is equivalent to having an $(\Aff(\R^{n+1}), \R^{n+1})$-structure on it. So $\mathfrak{F}$ inherits a tangential affine structure that varies smoothly. Moreover, the leaves of $\mathfrak{F}$ 
    admit a tangent parallel lightlike vector field $V$,
    and an induced parallel degenerate Riemannian metric with  radical $\R V$. This gives the leaves a $(\mathsf{L}_\u(n), \R^{n+1})$-geometry. 
\end{example}

It turns out that the (flat) connection preserved by $\L_\u(n)$ belongs to what is referred to as a distal connection in \cite{fried1986distality}. It is shown in this paper that such a connection is  complete when the manifold is compact. For the foliated version, we show: 
\begin{theorem}\label{Introduction: Theorem 1}
Let $(M,\mathfrak{F})$ be a compact $(\mathsf{L}_\u(n), \R^{n+1})$-foliated manifold.  The leaves of $\mathfrak{F}$ are $(\mathsf{L}_\u(n), \R^{n+1})$-complete. In particular, when the codimension is $0$, we obtain that a compact $(\mathsf{L}_\u(n), \R^{n+1})$-manifold is complete.\medskip
\end{theorem}

\noindent This gives  an affirmative answer to the \textbf{$(\mathsf{L}_\u(n), \R^{n+1})$-foliated completeness problem}.\\

One could also consider a ``dual'' structure to $(\L_\u(n),\R^{n+1})$, where the structure group is $$\L_\u^*(n):=\left\{\begin{pmatrix}
A & \beta \\
0 & 1
\end{pmatrix}\ltimes \R^{n+1} \ | \  \beta\in \R^{n}, \  A\in \O(n)\right\}.$$

\noindent The linear part of $\L_\u^*(n)$ is the dual of the linear part of $\L_\u(n)$, hence the name of dual geometry. It has an invariant parallel $1$-form (instead of an invariant parallel vector field), and an invariant foliation (instead of only a distribution). For this dual geometry, we establish an analogous statement to Theorem \ref{Introduction: Theorem 1}.

\begin{theorem}\label{Introduction: Theorem 1*}
Let $(M,\mathfrak{F})$ be a compact $(\mathsf{L}_\u^*(n), \R^{n+1})$-foliated manifold.  The leaves of $\mathfrak{F}$ are $(\mathsf{L}_\u^*(n), \R^{n+1})$-complete. 
\end{theorem}

\subsubsection{\textbf{Affine lightlike geometry}}\label{Intro-Subsection: Affine lightlike geometry} One could relax the $(\L_\u(n),\R^{n+1})$-geometry by assuming, instead of an invariant vector field, only an invariant line field.  
In this case, the structure group is 
$$\L(n):=\left\{\begin{pmatrix}
\lambda & \alpha^{\top} \\
0 & A 
\end{pmatrix}\ltimes \R^{n+1} \ | \ \lambda\in \R_{>0}, \  \alpha\in \R^{n}, \  A\in \O(n)\right\}.$$
We refer to this $(\L(n), \R^{n+1})$-geometry as the \textit{affine lightlike geometry}. \medskip

\begin{example}[weakly pp-waves are $(\mathsf{L}(n), \R^{n+1})$-foliated manifolds]\label{Intro-Example: weakly pp-wave}
    Let $(M,l)$ be a `weakly pp-wave' of dimension $n+2$, that is, a Lorentzian manifold with a null parallel line field $l$, for which the (totally geodesic) foliation $\mathfrak{F}$ orthogonal to $l$ is flat with respect to the induced connection. The leaves of $\mathfrak{F}$ have a natural affine lightlike geometry.  So weakly pp-waves give natural examples of $(\mathsf{L}(n), \R^{n+1})$-foliated manifolds.
\end{example}

Using essentially the same techniques as in Theorem \ref{Introduction: Theorem 1}, we characterize the completeness in the following way: 
\begin{theorem}\label{Introduction: Theorem 2}
    Let $(M, \mathfrak{F})$ be a compact $(\mathsf{L}(n), \R^{n+1})$-foliated manifold. The leaves of $\mathfrak{F}$ are $(\mathsf{L}(n), \R^{n+1})$-complete if and only if all geodesics tangent to the $\mathcal{V}$-leaves are affinely (or geodesically) complete, where $\mathcal{V}$ is the $1$-dimensional foliation tangent to the line field induced on any $\mathfrak{F}$-leaf by this structure.
\end{theorem}

Incompleteness may occur in Theorem \ref{Introduction: Theorem 2}: taking the product of a Hopf circle and a Euclidean torus  (Example \ref{Example: affine L-str}) defines such an example.
\subsubsection{\textbf{Pseudo-Riemannian variants.}}
A natural generalization of the above situation is the following pseudo-Riemannian variant of the $(\L_\u(n,0),\R^{n+1})$-geometry. Let $p+q=n$, and define $$\L_\u(p,q):=\left\{\begin{pmatrix}
1 & \alpha^{\top} \\
0 & A 
\end{pmatrix}\ltimes \R^{p+q+1} \ | \  \alpha\in \R^{p+q}, \  A\in \O(p,q)\right\},$$
where $\mathsf{O}(p,q)$ is the orthogonal group of the usual pseudo-Riemannian  scalar product on $\R^{p+q}$ of signature $(p,q)$. One can define similarly $\L(p,q)$, the pseudo-Riemannian analogue of $\L(n)$.\medskip

\begin{example}[Higher-index pp-waves]
 There are also natural examples of $(\L_\u(p,q), \R^{p+q+1})$-foliated manifolds, given by higher-index pp-waves of signature $(p+1,q+1)$. A higher-index pp-wave can be defined similarly to the Lorentzian case as a pseudo-Riemannian manifold of signature $(p+1,q+1)$, admitting a lightlike parallel vector field $V$, such that the distribution $V^\perp$ is flat with respect to the induced connection. As in the Lorentzian case, they admit Brinkmann local coordinates (the proof in \cite[Paragraph 4.2.4]{mzkundt} holds independently of the signature) of the form
\begin{align*}\label{Eq: pp-waves coordinates}
    2 du dv + H(u,x) du^2 + \sum_1^p dx_i^2 - \sum_1^q dx_i^2.
\end{align*}
Those spaces possess a natural tangential $(\L_\u(p,q), \R^{p+q+1})$-structure along the foliation tangent to $V^\perp$. In Lorentzian signature $q+1=1$, the isometries preserving the foliation act as automorphisms of the $\L(p,q)$-structure on the leaves, since in this case $V$ is the unique (up to scale) parallel null vector field tangent to the foliation. This is no longer true in higher index, and this is a fundamental difference between the Lorentzian and higher index pp-waves from the point of view of geometric structures. 
\end{example}
\begin{example}[A compact $\L_\u(1,1)$-foliated manifold]
    In dimension $4$, there is a special example of a symmetric pp-wave of index $2$,  given by the hyperbolic oscillator group $\Osc_s:=\R\ltimes \Heis_3$, where the $\R$-action acts trivially on the center $Z$ of $\Heis_3$,  and  by hyperbolic transformations on $\Heis_3/Z$. The hyperbolic oscillator has a bi-invariant metric $g_{\Osc_s}$ of signature $(2,2)$. Considering the space $(\Osc_s,g_{\Osc_s})$ with its full isometry group gives what we call the hyperbolic oscillator geometry. The latter geometry admits a unique (up to scale) parallel lightlike vector field $V$, and the orthogonal distribution $V^\perp$ is flat. So, by taking a cocompact lattice $\Gamma$ of  $\Osc_s$, one obtains a compact (locally symmetric) $\L_\u(1,1)$-foliated manifold $\Osc_s/\Gamma$. The $4$-dimensional symmetric pp-waves of index $2$ have been recently explored in \cite{Matea, kath2024pseudo}, where the authors compute the full isometry group and study the existence of compact complete manifolds modeled on them.
\end{example}

\medskip
The famous Markus conjecture states that an affine structure on a compact manifold is complete as soon as it admits a parallel volume form. The $(\L_\u(p,q), \R^{p+q+1})$ geometry and its dual are part of this conjecture, and therefore, compact $(\L_\u(p,q), \R^{p+q+1})$-manifolds (and their duals) are conjectured to be complete.
However, it turns out that the foliated completeness Problem \ref{Prob: foliated problem 1} fails in this case. We provide examples of compact $(\L_\u(p,q), \R^{p+q+1})$-foliated and $(\L_\u^*(p,q), \R^{p+q+1})$-foliated manifolds (for any $p, q \geq 1$) with incomplete leaves (Example \ref{Example: foliated L_u(1,1)-incomplete} and Example \ref{Example: foliated L_u*(1,1)-incomplete}). The issue arises from the existence of compact manifolds with foliations that admit  a tangential flat Lorentzian structure, such that the leaves are incomplete. Or, in other words, from the fact that there is no foliated version of Carri\`ere's theorem on the completeness of compact flat Lorentzian manifolds, as it is observed in Example \ref{Example: incomplete flat Lor}.
\medskip

For the dual geometry, with $q=1$, we show
\begin{proposition}
      Let $M$ be a compact $(\L_\u^*(n-1,1), \R^{n+1})$-manifold. Then the developing map is 
 injective; in particular, $M$ is Kleinian. If moreover, the fundamental group of $M$ is virtually solvable, then $M$ is complete.
\end{proposition}

\medskip

\paragraph{\textbf{Organization of the paper}} The organization of the paper is different from the order presented in the introduction. Section \ref{section: (G,X)-foliated completeness problem: affine unimodular lightlike geometry} deals with the $(G,X)$-foliated completeness problem  in Item \textbf{(B)}, in the case of the affine unimodular lightlike geometry and its dual. We prove Theorem \ref{Introduction: Theorem 1} and Theorem \ref{Introduction: Theorem 1*}. Then we study the case of the affine lightlike geometry in Section \ref{section: (G,X)-foliated completeness problem: affine lightlike geometry}: we prove Theorem \ref{Introduction: Theorem 2}, and give leafwise-incomplete examples. Section \ref{section: L_u(k,1) geometry and its dual} deals with the pseudo-Riemannian variants of the previous structures. In Section \ref{section: Foliated lightlike geometry: leafwise completeness}, we deal with the foliated completeness problem in Item \textbf{(A)}, and prove Theorem \ref{Theorem: completeness of lightlike tot geod foliation}. Finally, Section \ref{section: On completeness of compact Lorentzian manifolds with a null Killing field} is devoted to the study of the global completeness question in the presence of a null Killing field, focusing on dimension $3$ in Section \ref{section: Dimension 3: an example and a problem}. 
\bigskip

\paragraph{\textbf{Acknowledgment}} We would like to warmly thank Abdelghani Zeghib and Souheib Allout for the helpful discussions and valuable comments on this work. The first author is grateful to Ines Kath for her support, encouragement, and for useful discussions and remarks on early drafts of this paper.
The second author would like to warmly thank Miguel Sanchez Caja for the helpful discussions. Finally, we thank the referee for the valuable comments and suggestions that helped improve the quality of the presentation.\\
The second author is supported by the grants, PID2020-116126GB-I00\\
(MCIN/ AEI/10.13039/501100011033), and the framework IMAG/ Maria de Maeztu,\\ 
CEX2020-001105-MCIN/ AEI/ 10.13039/501100011033.

\section{$(G,X)$-foliated completeness problem: affine unimodular lightlike geometry}\label{section: (G,X)-foliated completeness problem: affine unimodular lightlike geometry}

\subsection{\textbf{$(G,X)$-structures.}}
Let $G$ be a Lie group, and $X$ a space on which $G$ acts transitively by real analytic diffeomorphisms. Then  $X=G/H$, where $H$ is a closed subgroup of $G$. A $(G,X)$-structure on a smooth manifold $M$ is an atlas on $M$, where the chart transitions are restrictions of elements of $G$. A manifold with a $(G, X)$-structure is called a $(G, X)$-manifold. A $(G, X)$-manifold $M$ gives rise to a local diffeomorphism $\Dev$, called the developing map, from the universal cover to the model space $X$, along with a representation $\rho: \pi_1(M) \to G$, called the holonomy representation (see  \cite[p. 100]{goldman2022geometric} or  \cite[p. 34]{thurston2022geometry}). The developing map and the holonomy representation satisfy the following equivariance property $$\Dev(\gamma x) =\rho(\gamma)\Dev(x),$$ for any $x\in \widetilde{M}$ and $\gamma \in \pi_1(M)$.
\begin{definition}[Completeness in the sense of $(G, X)$-structures]
    Let $M$ be a $(G, X)$-manifold. We say that $M$ is complete if the developing map $\Dev: \widetilde{M} \to X$ is a covering map.
\end{definition}
\begin{fact}\label{Fact: complete=complete}
    Let $M$ be $(G,X)$-manifold. Assume that $X$ is complete with respect to a $G$-invariant connection. Then $M$ is geodesically complete (with respect to the ``pullback'' connection) if and only if $M$ is $(G,X)$-complete.
\end{fact}
\subsection{Tangential $(G,X)$-foliation} We introduce here the notion of a $(G,X)$-foliated manifold. A smooth manifold $M$ with a foliation $\mathfrak{F}$ is referred to as a foliated manifold, denoted by the pair $(M,\mathfrak{F})$.

\begin{definition}[\textbf{Tangential $(G,X)$-foliation}]\label{Defi: GX foliation}
We say that a manifold $M$ has a tangential $(G,X)$-foliation if $M$ is a foliated manifold such that the leaves of $\mathfrak{F}$ are modeled on $(G,X)$. More precisely, for every $p\in M$, there exists an open neighborhood  $U$ and  a diffeomorphism (into the image) $\varphi: U\to  \R^k\times X$ such that every leaf of the foliation of $U$ (induced by $\mathfrak{F}$) is mapped to a level $(t,V_t)\subset \R^k\times X$. And if $(\varphi_i, U_i)$ and $(\varphi_j, U_j)$ are two such maps, the transition map $\varphi_{ij}:=\varphi_j\varphi_i^{-1}$ has the following form $(\psi(t), g_t(x))$ where $g_t\in G$ and $x \in V_t$ (compare \cite{furness1974affine, inaba1993tangentially, malakhal2002foliations}). 
\end{definition}
\medskip

\noindent \textbf{Convention:} We call a \textbf{$(G,X)$-foliated manifold} a manifold $M$ having a tangential $(G,X)$-foliation, as defined above.
\medskip

In addition to the examples provided in the introduction, we present here some natural examples of foliations with a tangential structure:  
 \begin{example}[Vector fields]
The data of a non-vanishing vector field on $M$ is equivalent to giving a $1$-dimensional foliation with a tangential  $(\R,\R)$-structure. The structure on the leaves is complete if and only if the vector field is complete. This is the case when $M$ is compact. 
 \end{example}
  \begin{example}[Lightlike foliations of Lorentzian surfaces]
   An orientable Lorentzian surface has two totally geodesic lightlike foliations. Each foliation naturally carries a tangential affine structure, modeled on $(\Aff(\R), \R)$. The local affine parameter is given by the geodesic parameter. The completeness of this affine structure is equivalent to the completeness of the lightlike geodesics tangent to the foliation. 
 \end{example}
  \begin{example}[A tangential Minkowski structure]
Consider $\R^2$ with the action of the hyperbolic matrix 
  $\varphi:=\begin{pmatrix}
         2& 1\\
         1 & 1
    \end{pmatrix}$ (this means that it has two distinct real eigenvalues). There is a Lorentzian flat metric on $\R^2$ preserved by $\varphi$ (take two $1$-forms $du$ and $dv$ defining the two eigendirections of $\varphi$, and set $g:=dudv$).   
    Define the flat $2$-torus $\T^2:=\R^2/\Z^2$, with $\overline{\varphi}$ the induced action of $\varphi$.  The mapping torus $\T_\varphi:= \T^2 \times [0,1]/(x,0)\sim (\overline{\varphi}(x),1)$ has a natural foliation with a tangential Minkowski structure, whose leaves are given by the  $\mathbb{T}^2$-fibers.
 \end{example}
\begin{example}[A tangential hyperbolic structure]
   Consider the group $\Sol\cong \R\ltimes \R^2$, with $\R$ acting on $\R^2$ by linear hyperbolic transformations
   $\begin{pmatrix}
       e^t & 0\\
       0 & e^{-t}
   \end{pmatrix}.$
   We have a short exact sequence $\Aff(\R)\to \Sol \to \R$. The left action of $\Aff(\R)$ on $\Sol$ defines a foliation $\mathfrak{F}$ with a tangential $(\Aff(\R), \Aff(\R))$-structure. 
   On the other hand, $\Aff(\R)$ acts simply transitively on the hyperbolic plane $\H^2$, inducing a natural hyperbolic structure on the leaves of $\mathfrak{F}$. The foliation is right invariant, so by taking a cocompact lattice $\Gamma$ of $\Sol$, we obtain  a compact manifold $\Sol/\Gamma$ with a tangential hyperbolic foliation. 
\end{example}
\begin{example}(Tangential $(G,G)$-foliations)
    This is defined similarly to the previous example. Let $L$ be a connected Lie group, and $G$  a connected closed subgroup of $L$. Then $L$ is naturally a $(G,G)$-foliated manifold. We obtain a compact $(G,G)$-foliated manifold if $L$ has a cocompact lattice. 

\end{example}
\noindent \textbf{Pullback of $G$-invariant tensor fields.} Consider $(M,\mathfrak{F})$ a $(G,X)$-foliated manifold. Let $T$ be a tensor field on $X$ which is $G$-invariant. One can associate a tensor field $T_M$ on $M$ in the following way. For any point $p \in M$, we take a chart around $p$ and we pullback the tensor field $T$ in a neighborhood of $p$. Since the transition maps, when restricted to a leaf, are restrictions of elements of $G$, and $T$ is $G$-invariant, we get a well defined tensor field $T_M$ on $M$. Moreover, by construction, $T_M$ is continuous on $M$. We will refer to $T_M$ as the ``pullback'' of $T$.
\begin{remark}[Leafwise completeness: tautological cases]\label{Remark: leafwise completeness trivial cases}
    There are two situations, one geometric and the other topological, in which any compact $(G,X)$-foliated manifold $(M, \mathfrak{F})$ is leafwise $(G,X)$-complete. The first situation is when the stabilizer in $G$ of some point in $X$ is compact; equivalently, $X$ admits a $G$-invariant Riemannian metric. In this case, taking the pullback of this metric defines a Riemannian metric on $T\mathfrak{F}$. Then, the completeness follows from Fact \ref{Fact: complete riemannian leaf}. The other situation is when all the leaves are simply connected, or more generally, when the holonomy representation of any leaf is trivial. Then $T\mathfrak{F}$ also supports a Riemannian metric, since the transition maps of the geometric structure are just restrictions of the identity map of $G$, and the leaves are therefore $(G,X)$-complete. 
\end{remark}

\subsection{\textbf{Affine unimodular lightlike geometry.}}
Recall that  the affine unimodular lightlike geometry is the one modeled on $(\mathsf{L}_\u(n), \R^{n+1})$, where $\mathsf{L}_\u(n)$ is the subgroup of $\Aff(\R^{n+1})$ preserving  the degenerate Riemannian metric $h:=dx_2^2+dx_3^2+ \ldots +dx_{n+1}^2$\, and the null vector field $V :=\partial_{x_1}$. In this subsection,  $V$ and $h$ will always refer to these $\L_\u(n)$-invariant objects. 
The foliation by the leaves of $V$ will be denoted by $\mathcal{V}$.

\subsection{$(\mathsf{L}_\u(n), \R^{n+1})$-foliated completeness}
This subsection is devoted to the study of the foliated completeness problem. We  show the following 
\begin{theorem}\label{Theorem 1}
Let $(M,\mathfrak{F})$ be a $(\mathsf{L}_\u(n), \R^{n+1})$-foliated compact manifold. Then the leaves of $\mathfrak{F}$ are $(\mathsf{L}_\u(n), \R^{n+1})$-complete.
\end{theorem}
The proof, presented in Paragraph \ref{subsection: Proof of Theorem 1}, is based on the theory of Riemannian submersions, developed by B. O'Neil in \cite{o1966fundamental} and A. Grey in \cite{gray1967pseudo}.  In the next paragraph, we will recall the main results needed for the proof.

\subsubsection{\textbf{Prerequisites on Riemannian submersions}}\label{subsection: Riem sub}

Let $(E, g)$ and $(B, g')$ be smooth Riemannian manifolds, where $\dim E \geq\dim B$. Let $\pi: E\to B$ be a smooth submersion, the regular value theorem implies that the fiber $\pi^{-1}(x)$ of any point $x\in B$ is a closed smooth submanifold of $E$ of dimension $\dim E -\dim B$. Let $p\in E$, we denote by $\mathcal{K}_p$ the tangent space at $p$ of the fiber of $\pi$, and $\mathcal{H}_p$ its orthogonal complement at $p$ with respect to the metric $g$. Thus, we can decompose the tangent space at $p$ as $T_pE= \mathcal{K}_p\oplus\mathcal{H}_p$. The distribution $\mathcal{K}$ is called the \textit{vertical} distribution, it is integrable, and the leaves are the fibers of $\pi$. The distribution $\mathcal{H}$ is called the \textit{horizontal} distribution; generically, it is far from being integrable. 
\begin{definition}
    Let $\pi: (E, g) \to (B, g')$ be a smooth submersion, we say that $\pi$ is a \textit{Riemannian submersion} if the differential of $\pi$ induces an isometry from the horizontal distribution to the tangent space of $B$. In other words, $d_p\pi$ is an isometry from $\mathcal{H}_p$ to $T_{\pi(p)}B$.
\end{definition}
Consider $\pi: (E, g) \to(B, g')$  a Riemannian submersion, and $\mathcal{H}$ the corresponding horizontal distribution. A vector $v\in T_pM$ is said to be horizontal if $v\in \mathcal{H}_p$. And a curve $\gamma: I\to E$ is said to be horizontal if for all $t\in \R$, $\gamma'(t)$ is a horizontal vector, i.e. $\gamma'(t)\in \mathcal{H}_{\gamma(t)}$.\medskip

We now state three facts that will be used throughout the section. 

\begin{fact}[O'Neil, \cite{pastore2004Riemannian} Proposition 1.10]\label{Fact: horizontal}
Let $\gamma: I\to E$ be a geodesic. If the tangent vector $\gamma'(t_o)$ is horizontal, then $\gamma$ is horizontal. In other words, the horizontal distribution is totally geodesic.
\end{fact}

\begin{fact}[Corollary 1.1 page 26, \cite{pastore2004Riemannian}]\label{Fact: geodesic}
Let $\pi: (E, g) \to (B, g')$ be a Riemannian submersion.
Consider $\gamma: I\subset \R \to E$ a parameterized horizontal curve with constant speed, i.e. for all $t\in I$ we have $\gamma'(t)\in \mathcal{H}_{\gamma(t)}$ and $g_{\gamma(t)}(\gamma'(t), \gamma'(t))=1$. Then, $\gamma$ is a geodesic in $E$ if and only if $\pi\circ \gamma$ is a geodesic in $B$.
\end{fact}

\begin{fact}[Proposition 3.2 \cite{hermann1960sufficient}]\label{Cor: Hermann Fibration }
 Let $\pi: (E, g) \to (B, g')$ be a Riemannian submersion. If the total space $E$ is complete, then $E$ is a fiber bundle over $B$. 
\end{fact}

\subsubsection{\textbf{Proof of Theorem $\ref{Theorem 1}$}}\label{subsection: Proof of Theorem 1}
In order to prove Theorem \ref{Theorem 1}, we need to show that the developing map $\Dev$ is a diffeomorphism from the universal cover of any leaf of $\mathfrak{F}$ onto the affine space $\R^{n+1}$. This will be done by constructing a particular Riemannian submersion, and applying 
Subsection \ref{subsection: Riem sub}. 

Let $(M,\mathfrak{F})$ be a $(\mathsf{L}_\u(n), \R^{n+1})$-foliated manifold. For simplicity, we will assume that the leaves are of  codimension $1$, i.e. $\dim M=n+2$ (although the proof works the same way for higher codimension). Recall that on $\R^{n+1}$ we have two $\mathsf{L}_\u(n)$-invariant objects: a vector field $V$ and a degenerate Riemannian metric $h$. Consider $V_M$ the pullback of $V$, it is a continuous non-vanishing vector field on $M$ tangent to $\mathfrak{F}$. Denote by $\mathcal{V}_M$ the foliation tangent to $V_M$. Consider also $h_M$ the degenerate metric given by the pullback of $h$ on the subbundle $T\mathfrak{F}$.  \\

Choose an arbitrary Riemannian metric $g_0$ on $M$, and decompose the tangent bundle $TM$ orthogonally with respect to $g_0$ as follows, $$TM= T\mathfrak{F}\oplus \R = \R  V_M\oplus\mathcal{S}\oplus \R Z,$$
where $\mathcal{S}$ is an $n$-dimensional distribution tangent to $\mathfrak{F}$ and transversal to $V_M$, and $Z$ is a $g_0$-unit vector field (it exists, up to taking a double cover).
\vspace{0.2cm}

\paragraph{\textbf{Constructing the Riemannian submersion.}}
We define a Riemannian metric $g$ on $M$, compatible with $h_M$ in the following sense $$g(V_M,V_M)=g(Z, Z)=1, \  g(V_M,Z )=0,$$ $$
g(V_M,X)=0,\ g(Z,X)=0, \ g(X, Y)=h_M(X,Y)$$
for all $X, Y\in \mathcal{S}$. The metric $g$ is indeed Riemannian, since the restriction of $h_M$ to the distribution $\mathcal{S}$ is non-degenerate and Riemannian (recall that $\mathcal{S}$ is transversal to $V_M$, which spans the radical of $h_M$).\medskip

\begin{fact}\label{Fact: complete riemannian leaf}
    Let $(M, \mathfrak{F}, g)$ be a compact foliated Riemannian manifold. Then any leaf $F$ with the induced Riemannian metric is geodesically complete.
\end{fact}

\begin{corollary}\label{cor: geodesically-complete}
 Let $F$ be a leaf of $\mathfrak{F}$. Then $(F, g_{|_{F}})$ is geodesically complete.    
\end{corollary}
\begin{proof}
This follows from Fact \ref{Fact: complete riemannian leaf}. 
\end{proof}

Let $(\widetilde{F}, \widetilde{g_{|_{{F}}}})$ be the universal cover of $F$, endowed with the lifted Riemannian metric. The Riemannian manifold $(\widetilde{F}, \widetilde{g_{|_{{F}}}})$ is geodesically complete. Let  $\Dev: \widetilde{F} \to \R^{n+1}$ be a developing map for the leaf $F$. Let  $\mathfrak{p}: \R^{n+1}\to \R^{n+1}/\mathcal{V}\cong \R^{n}$ be the projection onto the space of $\mathcal{V}$-leaves. The latter is identified to $\R^{n}$, equipped with the flat Riemannian metric given by the projection of $h$ by $\mathfrak{p}$. We denote the projected metric by $\Bar{h}$.  Define $\pi:= \mathfrak{p} \circ \Dev$, this is a smooth submersion from $\widetilde{F}$ to $B:=\pi(\widetilde{F}) \subset \R^n$, an open subset of $\R^n$.
\begin{lemma}
    The map $\pi: (\widetilde{F}, \widetilde{g_{|_{{F}}}}))\to (B, \Bar{h}_{|_{B}}) $ is a Riemannian submersion.
\end{lemma}
\begin{proof}
  The horizontal distribution on $\widetilde{F}$ is given by $\widetilde{\mathcal{S}}$, the lift of $\mathcal{S}_{\vert F}$ to $\widetilde{F}$. Let $x\in \widetilde{F}$ and let $y:=\Dev(x)$. Then, $E_{y}:=d_x\Dev(\widetilde{\mathcal{S}}_x)$ is a vector subspace of $T_y \R^{n+1}$ transversal to the radical of $h_{y}$ (i.e. to $\R V$). By construction, the  metric induced on $\widetilde{\mathcal{S}}_x$ by $g$ is isometric to the metric induced on $E_{y}$ by $h$. The lemma follows then from the fact that the metric induced by $h$ on any vector subspace of $T_{y}\R^{n+1}$ transversal to $V$ is isometric to $(T_{y}\R^{n+1}/\R V, \Bar{h}_{\mathfrak{p}(y)})$.
\end{proof}

\begin{corollary}
$(B, \Bar{h}_{|_{B}})$ is complete, in particular, we have $B=\R^{n}$.   
\end{corollary}

\begin{proof}
Let $\gamma : I \to \widetilde{F}$ be a horizontal geodesic (recall that the horizontal distribution is totally geodesic by Fact \ref{Fact: horizontal}). Then $\gamma$ is complete by Corollary \ref{cor: geodesically-complete}. Moreover, by Fact \ref{Fact: geodesic}, the projection of $\gamma$ by $\pi$ is a geodesic in the base space, which is an open subset of the flat euclidean space, hence $\pi\circ\gamma$ is an affine segment. Since $\gamma$ is complete, this affine segment is a complete affine line. This is true for any horizontal geodesic in all possible directions, hence $B=\R^n$. 
\end{proof}

\begin{corollary}\label{Cor:fibre bundle}
    The map $\pi: (\widetilde{F}, \widetilde{g_{|_{{F}}}}))\to (\R^{n}, \Bar{h})$ is a trivial topological fiber bundle, i.e. we have $\widetilde{F} \cong \R^{n}\times l $, with $l=\pi^{-1}(0)$.  
\end{corollary}
\begin{proof}
    By  Fact \ref{Cor: Hermann Fibration }, we have that $\pi: (\widetilde{F}, \widetilde{g_{|_{{F}}}}))\to (\R^{n}, \Bar{h})$ is a topological fiber bundle. Since the base space is contractible, the bundle is trivial (a product bundle).
\end{proof}
We are now able to prove the main theorem of this section, Theorem \ref{Theorem 1}. 

\begin{proof}[Proof of Theorem \ref{Theorem 1}]
Let $F$ be a leaf of $\mathfrak{F}$, with $\Dev: \widetilde{F} \to \R^{n+1}$ a developing map. We will show that $\Dev$ is a diffeomorphism onto $\R^{n+1}$. We start by the surjectivity.  Recall that the  Riemannian submersion $\pi=\mathfrak{p} \circ \Dev: \widetilde{F} \to \R^n $, where $\R^n$ is the Euclidean space, is surjective. Let $z\in\R^{n+1}$ be any point,  and let $x_o \in\widetilde{F} \smallsetminus \pi^{-1}(\{\mathfrak{p}(z)\})$. Set $x:=\pi(x_o)$. Consider the line segment $\gamma :=[x,\mathfrak{p}(z)]$ in $\R^n$, and let $\Tilde{\gamma}$ be the horizontal lift  of $\gamma$ by $\pi$, starting at $x_o$. Denote by $z_o$ the other endpoint of $\Tilde{\gamma}$ in $\widetilde{F}$. 
Since $\pi(z_o)=\mathfrak{p}(z)$, we see that $\Dev(z_o)\in \mathcal{V}(z)$, where $\mathcal{V}(z)$ denotes the $\mathcal{V}$-leaf containing $z$. Moreover, the developing map sends the $\widetilde{\mathcal{V}}_M$-leaf containing $z_o$ to the $\mathcal{V}$-leaf containing $z$, i.e. $\Dev(\widetilde{\mathcal{V}}_{M}(z_o))\subset \mathcal{V}(z)$. We claim that the restriction of $\Dev$ to any $\widetilde{\mathcal{V}}_M$-leaf is injective. Indeed, since the vector field $\widetilde{V}_M$ is  the pullback of $V$ via $\Dev$, it follows that $\Dev$ is equivariant with respect to their respective flows, i.e. for any $\Tilde{p} \in \widetilde{F}$ we have $$\Dev(\phi_{\widetilde{V}_{M}}^{t}(\Tilde{p})) = \phi_{V}^{t}(\Dev(\Tilde{p})).$$
And now, since the flow of $V$ acts freely (by translations) on the affine space $\R^{n+1}$, and in particular, on its orbits (affine lines), the claim follows. On the other hand, since $M$ is compact, the flow of $\widetilde{V}_M$ is complete. Therefore, by the equivariance formula, the restriction of the developing map to any $\widetilde{\mathcal{V}}_M$-leaf is also surjective. The surjectivity of $\Dev$ follows.\\ 
For the injectivity, assume that there are two distinct points $x_o, y_o\in \widetilde{F}$ (necessarily in different $\widetilde{\mathcal{V}}_M$-leaves) such that $\Dev(x_o)=\Dev(y_o)$. Hence,  $\pi(x_o)=\pi(y_o)=x$, which means that the fiber of $x$ contains at least two $\widetilde{\mathcal{V}}_M$-leaves. However, Corollary \ref{Cor:fibre bundle} together with the fact that $\widetilde{F}$ is connected imply that the fiber is necessarily connected, leading to a contradiction. 
\end{proof}
Given a  $(\mathsf{L}_\u(n), \R^{n+1})$-foliated manifold $(M, \mathfrak{F})$, the pullback to $T \mathfrak{F}$ of the flat affine connection of $\R^{n+1}$ defines a flat affine connection on the leaves of $\mathfrak{F}$. And we have
\begin{corollary}\label{Cor: L_u complete implies geod complete}
   Let $(M,\mathfrak{F})$ be a compact $(\mathsf{L}_\u(n), \R^{n+1})$-foliated manifold. Then the leaves of $\mathfrak{F}$ are geodesically complete with respect to the induced flat affine connection.
\end{corollary}
\begin{proof}
    
    This is a consequence of Theorem \ref{Theorem 1} and Fact \ref{Fact: complete=complete}.
\end{proof}

For the dual geometry, we need the following fact:
\begin{fact}[Proposition $8$ \cite{leistner2016completeness}]\label{Fact: structure of the universal cover}
Let $M$ be a compact manifold endowed with a  nowhere vanishing closed $1$-form $\omega$. Let $\widetilde{\omega}$ be the pullback of $\omega$ to the universal cover $\widetilde{M}$, and $Z$ a vector field on $\widetilde{M}$ such that $\widetilde{\omega}(Z)=1$. Then $\widetilde{M}$ is diffeomorphic to $\R \times \widetilde{F}$ where $\widetilde{F}$ is a leaf of the foliation $\mathfrak{F}_{\widetilde{M}}$ tangent to $\ker(\widetilde{\omega})$. And the flow of $Z$ acts simply transitively on the space of leaves, identified to $\R$, and maps a leaf to another diffeomorphically. 
\end{fact}

\begin{corollary}[The dual geometry]\label{Cor: Lu*}
    Let $(M,\mathfrak{F})$ be a compact $(\mathsf{L}_\u^*(n), \R^{n+1})$-foliated manifold. Then the leaves of $\mathfrak{F}$ are $(\mathsf{L}_\u^*(n), \R^{n+1})$-complete.
\end{corollary}

\begin{proof}
In the basis $(e_1,\ldots,e_{n+1})$ of $\R^{n+1}$ where the matrix group $\mathsf{L}_\u^*(n)$ is represented, the $1$-form $\omega:=e_{n+1}^*$ as well as the codimension $1$ foliation $\mathfrak{G}$ tangent to $\ker \omega$ are $\mathsf{L}_\u^*(n)$-invariant. 
Denote by $\mathfrak{G}_M$ the foliation on $M$, tangent to $\mathfrak{F}$, given by the pullback of  $\mathfrak{G}$.  And denote by $\omega_M$ the $1$-form on $M$ given by the pullback of $\omega$. 
There exists a vector field $W$ on $M$ such that $\omega_M(W)=1$. In fact, $W$ can be chosen as a global section of a tangent affine subbundle with contractible fibers, ensuring the existence of the section.
We denote by $\widetilde{\omega}_M$, $\widetilde{W}$, $\widetilde{\mathfrak{G}_M}$ the lifts of the corresponding objects to the universal cover $\widetilde{M}$. 
We denote the flow of a vector field $U$ by $\phi_U^t$. Since $M$ is compact, $W$ is complete. Fix a leaf $F$ of $\mathfrak{F}$, and consider a connected component $\widetilde{F}$ of its lift to $\widetilde{M}$. We have a developing map $\Dev: \widetilde{F} \to \R^{n+1}$. 
By Fact \ref{Fact: structure of the universal cover}, the flow of $\widetilde{W}$, restricted to $\widetilde{F}$, acts simply transitively on the space of $\widetilde{\mathfrak{G}_M}_{\vert \widetilde{F}}$-leaves, and maps one leaf diffeomorphically to another.
And since $\omega$ is $\mathsf{L}_\u^*(n)$-invariant, we have the following equivariance formula $$\Dev \circ \phi_{\widetilde{W}}^t = \phi_{e_{n+1}}^t \circ \Dev.$$ 
On the other hand, we can show similarly to the proof of Theorem \ref{Theorem 1} that $\Dev$, restricted to a leaf of $\widetilde{\mathfrak{G}_M}$ through $p \in \widetilde{F}$, is bijective onto the leaf of $\mathfrak{G}$ through $\Dev(p) \in \R^{n+1}$. Fix  a leaf of $\widetilde{\mathfrak{G}_M}$ in $\widetilde{F}$. The equivariance formula applied to this leaf shows that the developing map is surjective, thus, bijective on all $\widetilde{F}$, which ends the proof. 
\end{proof}

\subsection{Applications} In what follows, we discuss applications of Theorem \ref{Theorem 1}.

\subsubsection{Compact pp-waves.} 
The special class of Lorentzian manifolds, known as pp-waves, play a significant role in General Relativity.
It was proved in \cite{leistner2016completeness} that a compact pp-wave is complete. A key step in the proof was establishing the completeness of the leaves of the codimension $1$ totally geodesic foliation $\mathfrak{F}$ tangent to $V^\perp$. As observed in Example \ref{Intro-Example: pp-waves}, the leaves of $\mathfrak{F}$ inherit a natural $(\mathsf{L}_\u(n), \R^{n+1})$-structure, compatible with the connection induced from the Lorentzian metric.  We then recover, applying Corollary  \ref{Cor: L_u complete implies geod complete}, the completeness result along the leaves of $\mathfrak{F}$. We see that the completeness of the $\mathfrak{F}$-leaves of a pp-wave is independent of the presence of a Lorentzian ambient metric, and  is, in fact, a purely leafwise completeness phenomenon! 

\subsubsection{Compact manifolds modeled on (non-flat) homogeneous plane waves} Within the family of pp-waves, there exists a distinguished subclass of metrics known as plane waves, defined by specific conditions on the Riemann curvature  tensor. These metrics are widely studied due to their very interesting mathematical properties. 

By a compact model of a homogeneous plane wave, we mean a compact $(G_{\rho}, X_{\rho})$-manifold, where $X_{\rho}$ is a non-flat simply connected homogeneous plane wave, and $G_{\rho}$ is the identity component of its isometry group. These groups are described in \cite{article1}.   For $\dim X_\rho = n+2$, we have
$$G_\rho = (\R \times K) \ltimes \Heis, \;\; X_\rho=G_\rho / I, \, I \subset K \ltimes \Heis,$$
where $K \ltimes \Heis \subset \mathsf{L}_\u(n)$. We denote by $V$ the parallel null vector field of $X_\rho$, and by $\mathfrak{F}$ the foliation tangent to $V^\perp$. 
The subgroup of $G_\rho$ fixing a leaf of $\mathfrak{F}$ is exactly $K \ltimes \Heis$.

Non-flat homogeneous plane waves are divided into two families: one where the model space $X_{\rho}$ is geodesically complete and the parallel null vector field $V$ is preserved by $G_\rho$, and another where $X_{\rho}$ is incomplete and only the line field $\R V$ is preserved by $G_\rho$. 
\begin{theorem}
A compact manifold $M$ modeled on a non-flat simply connected homogeneous plane wave $(G_\rho, X_\rho)$ is $(G_\rho, X_\rho)$-complete. 
\end{theorem}

\begin{proof}
If $M$ is modeled on the first family, it admits a plane wave structure. To see this, consider the pullback of $V$ and the Lorentzian metric $g$ on the model space $X_\rho$ by the developing map $\mathsf{Dev}: \widetilde{M} \to X_\rho$. Since both $V$ and $g$ are invariant under the action of $G_\rho$, this induces a Lorentzian metric and a vector field on the universal cover $\widetilde{M}$, both invariant under the deck transformations by $\pi_1(M)$. Thus, $M$ inherits a Lorentzian structure, with a parallel null vector field, locally isometric to that of $X_\rho$, establishing that $M$ is a plane wave. 
This manifold is complete by \cite{leistner2016completeness, mehidi2022completeness}. 
In contrast, if $M$ is modeled on the second family, it is not a plane wave but rather a weakly plane wave. It has a parallel null line field $l$ instead of a parallel null vector field. The distribution $l^\perp$ is tangent to a lightlike totally geodesic foliation $\mathfrak{F}_M$. Even if $V$ is not preserved by the whole group $G_\rho$, it appears from the description of $G_\rho$ that the subgroup fixing a leaf is a subgroup of the affine unimodular group $\L_\u(n)$,\, $n+2=\dim X_\rho$. Thus,
the holonomy representation of an $\mathfrak{F}$-leaf takes values in $\L_\u(n)$. Consequently, $(M,\mathfrak{F}_M)$ is an $(\L_\u(n), \R^{n+1})$-foliated manifold. We claim that $M$ is $(G_\rho, X_{\rho})$-complete, despite being geodesically incomplete (the geodesic incompleteness comes from the model space!). Indeed, by Theorem \ref{Theorem 1}, each leaf in the universal cover $\widetilde{M}$ is mapped bijectively to a leaf in the model via the developing map. On the other hand, there is a $G_{\rho}$-invariant closed $1$-form on $X_\rho$, whose kernel is equal to $V^\perp$. 
Indeed, the space of $\mathfrak{F}$-leaves is affinely identified with the $\R$-factor in $G_\rho$, on which $G_\rho$ acts by affine transformations: $K\ltimes\Heis$ acts trivially, while the $\R$-factor acts by translation. 
As a result, there is a closed $G_\rho$-invariant $1$-from $\omega_0$ on the space of leaves. Pulling back $\omega_0$ via the projection map $X_{\rho}\to X_{\rho}/\mathfrak{F}$ yields the desired $1$-form on $X_\rho$.    As in the proof of Corollary \ref{Cor: Lu*}, this allows to define a flow on $\widetilde{M}$ that acts transversely on the foliation, mapping leaves to one another equivariantly with respect to the developing map. We conclude as in the proof of Corollary \ref{Cor: Lu*}.    
\end{proof}

\section{$(G,X)$-foliated completeness problem: affine lightlike geometry}\label{section: (G,X)-foliated completeness problem: affine lightlike geometry}
We relax the geometric structure $(\mathsf{L}_\u(n), \R^{n+1})$ by assuming that only the direction of the vector field $V:=\partial_{x_1}$ is preserved by the structure group. 
In other words, the data is an invariant line field rather than an invariant vector field. The geometry in this case is the affine lightlike geometry $(\mathsf{L}(n),\R^{n+1})$.
We denote again by $\mathcal{V}_M$ the foliation tangent to the line field induced on $M$ by this structure. 

We now ask whether a completeness result analogous to Theorem \ref{Theorem 1} holds for this geometric structure $(\L(n), \R^{n+1})$. 
The methods from the previous case of the $(\mathsf{L}_\u(n), \R^{n+1})$-geometry readily show that the developing map $\Dev: \widetilde{F} \to \R^{n+1}$ on the universal cover of some leaf $F$ is injective, and surjective from $\widetilde{F}/\widetilde{\mathcal{V}}_M$ onto the space of $\mathcal{V}$-leaves.  However, surjectivity may fail along the $\mathcal{V}_M$-leaves. This failure is due to the absence of an invariant vector field. Indeed, the key ingredient for surjectivity along the $\mathcal{V}_M$-leaves was the existence of a  vector field $V_M$, whose flow was both equivariant and complete. Formulating this observation we get the following

\begin{theorem}\label{Theorem 2: completeness with L-structure}
    Let $(M,\mathfrak{F})$ be a compact $(\L(n), \R^{n+1})$-foliated manifold. A leaf of $\mathfrak{F}$ is complete if and only if the geodesics tangent to $\mathcal{V}_M$ are complete. 
\end{theorem}

It turns out that the incomplete situation can occur. In what follows, we give some examples illustrating 
this. We present two kind of examples: the first kind is an affine manifold with a codimension $1$ foliation that supports an $(\L(n), \R^{n+1})$-structure, and the second kind is a ``weakly pp-wave''.
\begin{example}[$\mathsf{Hopf^{\lambda}_1}\times \T^2$]\label{Example: affine L-str}
Let $\lambda >0$. Consider the open subset of $\R^3$ given by $U:=\{(v,x,u), v >0\}$, and take the quotient of $U$ by $\Gamma$ generated by 
\begin{align*}
 \gamma_1&: (v,x,u) \to (\lambda v,x,u)\\
 \gamma_2&: (v,x,u) \to (v,x+1,u)\\
 \gamma_3&: (v,x,u) \to (v,x,u+1).
\end{align*} 
Here, $\Gamma$ is a subgroup of $\Aff(\R^3)$ acting properly discontinuously and freely on $U$, with a compact quotient. So $\Gamma \backslash U$ is a compact affine manifold. It is the product of a Euclidean torus and an Hopf circle of affine length $\lambda$. 
Now, equip $\R^3$ with the Lorentzian flat metric $2du dv + dx^2$. The group $\Gamma$ does not preserve the metric, but it preserves the null direction $\R V$, where $V=\partial_v$, together with the degenerate Riemannian metric $dx^2$ on the leaves of $V^\perp$. Hence, the foliation tangent to $V^\perp$ induces a $(\L(1),\R^2)$-foliation on $\Gamma \backslash U$.  Note that if $\Gamma$ is a subgroup of the Poincar\'e group, i.e.  preserves the Lorentzian flat metric, then $U$ is necessarily $\R^3$ and the quotient is complete due to Carri\`ere's result. We see that in this case, having a Lorentzian ambient structure is not artificial regarding the foliated completeness question. 
\end{example}
\begin{example}[\textbf{Clifton-Pohl torus}]\label{Example: CP}
Let $(T_{CP}, \Bar{g}):=(\R^2 \smallsetminus \{(0,0)\}, g:=\frac{2dxdy}{x^2+y^2})/ (x,y) \sim 2 (x,y)$ denote the Clifton-Pohl torus. Let $\mathcal{N}$ be one of its two null foliations. Then, the tangent bundle of $\mathcal{N}$ defines a parallel null line field on $T_{CP}$. It is known that the leaves of the foliation are incomplete. And $\mathcal{N}$ has a natural affine structure, which is in fact a homothety structure $(\R^*_+, \R)$.
\end{example}

\begin{example}[\textbf{Higher dimensions}]\label{Example: CP higher dimension}
\textbf{(1)} Consider the product $M=T_{CP} \times \mathbb{S}^1$, with Lorentzian metric $\Bar{g} + \overline{dt^2}$.  Denote again by $\mathcal{N}$ one of the null foliations of  $T_{CP}$.  
Consider the lightlike and totally geodesic foliation $\mathfrak{F}:= \mathcal{N} \times \mathbb{S}^1$ of $M$. The line field $T\mathfrak{F}^\perp$ on $M$ is null and parallel. The leaves of $\mathfrak{F}$ are flat (with respect to the induced connection) and incomplete. So $M$ is a weakly pp-wave having incomplete leaves. 
A leaf of $\mathfrak{F}$ is either a cylinder or a torus. And the universal cover of a leaf develops onto $\R \times I$, where $ I$ is some open interval in $\R$ which is either affinely equivalent to $]0,1[$ or to $]0,+\infty[$ (both cases occur). 

Incompleteness occurs exactly along the null geodesics tangent to $\mathcal{N} \times \S^1$, and along spacelike geodesics not tangent to the $\mathbb{S}^1$-factor (a flat band of the form $\R \times I \subset \R^2$, with $I$ as above, contains spacelike lines that leave the band). 

\textbf{(2)} The previous example can be modified in order to have all $\mathfrak{F}$-leaves non-compact. Consider $X=\R^2 \smallsetminus \{(0,0)\} \times \R$ with Lorentzian metric $g + dt^2$. Its isometry group, up to finite index, is  $G=\R \times \R$, where the first $\R$-factor acts by a $1$-parameter group of homotheties on $(\R^2 \smallsetminus \{(0,0)\}, g)$, and the second $\R$-factor acts by translation on $(\R, dt^2)$. The group $G$ preserves the foliation $\mathcal{N} \times \R$ of $X$. Now, consider a discrete subgroup $\Gamma \subset G$ acting properly freely and cocompactly on $X$, and not intersecting the second $\R$-factor. The quotient $\Gamma \backslash X$ is a Lorentzian torus, where the induced foliation consists of leaves that are either planes or cylinders. By Remark \ref{Remark: leafwise completeness trivial cases}, one cannot find a quotient space where all the leaves of the foliation are simply connected. 

\textbf{(3)} Taking the product of $T_{CP}$ with a compact Riemannian manifold (of higher dimension) gives an example with non-flat $\mathfrak{F}$-leaves. 
\end{example}

\section{Pseudo-Riemannian variants and their duals}\label{section: L_u(k,1) geometry and its dual}

In this section, we will see that leafwise completeness may fail for the $(\L_\u(p,q), \R^{p+q+1})$-structure and its dual. We will construct explicit examples of incomplete foliations. But before that, we first consider the case of a codimension $0$ foliation, i.e. when the structure is defined on a compact manifold. We restrict to the case $q=1$.

\subsection{Compact $(\L_\u(n,1), \R^{n+2})$-manifolds}
Completeness of compact $(\L_\u(n,1), \R^{n+2})$-manifolds is an open question and a natural part of the Markus conjecture, as mentioned in the introduction. We also note that the method in \cite{carriere1989autour} does not apply here, even for $n=1$, due to the fact that the discompacity of $\L_\u(n,1)$ is $2$.

\subsection{Compact $(\L_\u^*(n,1), \R^{n+2})$-manifolds}
The dual structures were mentioned earlier by Carri\`ere in \cite[Theorem 3.2.1]{carriere1989autour}. Let $M$ be a compact $(\L_\u^*(n,1), \R^{n+2})$-manifold. This structure on $M$ gives rise naturally to a codimension $1$ totally geodesic Lorentzian foliation $\mathfrak{S}$ with flat leaves.  In \cite[Theorem 3.2.1]{carriere1989autour}, it is shown that the developing map is injective on each $\mathfrak{S}$-leaf. Here, we can use the extra invariant $1$-form to show that the developing map of $M$ is itself injective. 
\begin{fact}\label{Fact: d inj}
    Let $M$ be a compact $(\L_\u^*(n,1), \R^{n+2})$-manifold. The developing map $\Dev: \widetilde{M} \to \R^{n+2}$ is injective. As a result, the holonomy representation is injective, and the fundamental group is 
    isomorphic to a discrete subgroup of $\L_\u^*(n,1)$ acting properly freely on $\Dev(\widetilde{M})$.
\end{fact}
\begin{proof}
We already know that $\Dev$ is injective on a $\mathfrak{S}$-leaf (\cite[Theorem 3.2.1]{carriere1989autour}).  Global injectivity of $\Dev$ follows identically to the second part of the proof of Corollary \ref{Cor: Lu*}. The rest is an immediate consequence.
\end{proof}
 
\begin{corollary}[Solvable implies complete]\label{Cor: dual geometry, solvable implies complete}
    Let $M$ be a compact $(\L_\u^*(n,1), \R^{n+2})$-manifold with a virtually solvable holonomy group. Then $M$ is complete.
\end{corollary}
\begin{proof}
    First, observe that $M$ is aspherical. Indeed, $\widetilde{M}$ fibers over $\R$ with a contractible fiber (the fibers are the leaves of $\mathfrak{S}$, which are contractible by \cite[Proposition 3.2.2]{carriere1989autour}). This yields $\mathsf{cd}(\pi_1(M))=\dim(M)$, where $\mathsf{cd}(\pi_1(M))$ is the cohomological dimension of $\pi_1(M)$. On the other hand, by Fact \ref{Fact: d inj} we get that $\pi_1(M)$ injects into a discrete subgroup of the connected Lie group $\L_\u^*(n,1)$. However, by assumption $\pi_1(M)$ is virtually solvable, then by \cite[Lemma 2.1]{milnor1977fundamental} $\pi_1(M)$ is even virtually polycyclic. In particular, we get  $\mathsf{Rank}(\pi_1(M))=\mathsf{cd}(\pi_1(M))=\dim(M)$. Since $M$ admits a parallel volume form, the claim follows from \cite[Corollary 3.8]{goldman1986affine}.
\end{proof}
\begin{corollary}
     Let $M$ be a compact $(\L_\u^*(1,1), \R^{1+2})$-manifold. Then $M$ is complete. 
\end{corollary}
\begin{proof}
     Fact \ref{Fact: d inj} implies that the holonomy representation is injective. Hence, $\pi_1(M)$ is isomorphic to a subgroup of $\L_\u^*(1,1)$ . Clearly, $\L_\u^*(1,1)$ is solvable. So, the claim follows from Corollary \ref{Cor: dual geometry, solvable implies complete}.
\end{proof}
\subsection{The foliated problem} Here, we construct examples where leafwise completeness fails for the aforementioned geometries. 

 \begin{example}[\textbf{Compact manifold with incomplete flat Lorentzian leaves}]\label{Example: incomplete flat Lor}
Consider $\SOL= \R \ltimes \R^2 $, with its Lie brackets $[A,X]=X$, $[A,Y]=-Y$.  Endow it with the left invariant Lorentzian metric given in the bases $(A, X, Y)$ by
$$q=\begin{pmatrix}
0 & 1 & 0  \\
1 & 0 & 0  \\
0 & 0 & 1 
\end{pmatrix}.$$
This is in fact a (non-flat) plane wave metric on $\SOL$ $($see \cite[Theorem 1.4]{allout2022homogeneous}$)$.
The left invariant distribution induced by $\Span(A,X)$ is integrable, and a simple computation shows that it is flat, so that the induced metric is a left invariant Lorentzian flat metric on $\Aff(\R)$, which is then incomplete.  In fact, the above metric on $\Sol$ is given explicitly in coordinates in \cite[Theorem 1.4]{allout2022homogeneous}, and the incomplete flat foliation there corresponds to the levels of the $x$-coordinate. Now, consider the left quotient of $\SOL$ by some lattice $\Gamma$. The left action of $\Gamma$ preserves this foliation, which is then well defined in the compact quotient $\Gamma \backslash \Sol$. Its leaves are incomplete, and develop into the half Minkowski plane. This gives an example of an incomplete foliated $(G,X)$-structure, where the geometry here is the flat Lorentzian structure.  It is known that this structure is complete when the manifold is compact (Carri\`ere theorem). This example shows in particular that a foliated Carri\`ere theorem is not true. 
\end{example}

\begin{example}[\textbf{Compact incomplete  $(\L_\u(p,q),\R^{p+q+1})$-foliated manifold}]\label{Example: foliated L_u(1,1)-incomplete}
Consider the product $\SOL \times \R$, with the basis $(A,X,Y,V)$ of the Lie algebra, where $V$ is the basis for the $\R$-factor. Define a left invariant Lorentzian metric on it such that: $\SOL$ is endowed with the metric given in the previous example, $\R$ is Riemannian, and the two spaces are orthogonal. So in the basis $(A,X,Y,V)$, we have $$q=\begin{pmatrix}
0 & 1 & 0 &0 \\
1 & 0 & 0  &0\\
0 & 0 & 1 &0\\
0&0&0&1
\end{pmatrix}.$$ The leaves of the foliation tangent to the left invariant distribution induced by $\Span(A,X,V)$ are flat Lorentzian, and they admit a parallel spacelike vector field $V$. Hence, the leaves are modeled on the $(G,X)$-structure, where the group structure in the basis $(V,A,X)$ is given by $$G=\left\{\begin{pmatrix}
1 & 0 & 0  \\
0 & e^{t} & 0  \\
0 & 0 & e^{-t}
\end{pmatrix}\ltimes \R^3 \  | \ t\in \R\right\}.$$ This is a subgeometry (stiffened) of the $(\L_\u(1,1), \R^3)$-geometry. Now, taking a cocompact lattice of the form $\Gamma\times \Z \subset \SOL\times \R$ gives a compact, locally homogeneous, $(\L_\u(1,1), \R^3)$-foliated manifold. And  the previous example shows that the leaves are incomplete. 
By taking products of the previous example, and potentially including a flat Riemannian torus, we can construct a compact $\L_\u(p,q)$-foliated manifold with incomplete leaves for any $p,q\geq1$. 
Note that here, the holonomy representation of the fundamental group of any leaf preserves a Lorentzian metric, since it is a subgroup of $\O(1,1)\ltimes \R^3 \subset \L_\u(1,1)\ltimes \R^3$ (with a non-trivial projection to the $\O(1,1)$-factor). Observe also that the leaves develop injectively, so the structure is Kleinian!
\end{example}

\begin{remark}[\textbf{Compact incomplete $(\L_\u^*(p,q),\R^{p+q+1})$-foliated manifold}]\label{Example: foliated L_u*(1,1)-incomplete}
    The example above shows that Problem $\ref{Prob: foliated problem 1}$ also fails for the dual $(\L_\u^*(1,1), \R^3)$-geometry. Indeed, up to permuting the basis, the structure group becomes $$\left\{\begin{pmatrix}
e^t & 0 & 0  \\
0 & e^{-t} & 0  \\
0 & 0 & 1 
\end{pmatrix}\ltimes \R^3 \ | \ t\in \R\right\},$$ which is also a subgeometry of $(\L_\u^*(1,1), \R^3)$. From this example, one can obtain in the same way as in Example \ref{Example: foliated L_u(1,1)-incomplete} a compact $\L_\u^*(p,q)$-foliated manifold with incomplete leaves for any $p,q\geq1$. 
\end{remark}

\section{Foliated lightlike geometry: leafwise completeness}\label{section: Foliated lightlike geometry: leafwise completeness}

\subsection{Case of a totally geodesic lightlike foliation of a Lorentzian manifold}
Let $(M,g,\mathfrak{F})$ be a compact Lorentzian manifold with a totally geodesic lightlike foliation $\mathfrak{F}$. The metric induced on each leaf is a Riemannian degenerate metric. 
We denote by $\mathcal{N}$ the ($1$-dimensional) foliation tangent to the radical. The goal of this section is to prove the following result, which corresponds to the situation described in Item \textbf{(A)} of the introduction.

Let us mention that the foliation $\mathcal{N}$, when restricted to each leaf $F$ of $\mathfrak{F}$, is transversely  Riemannian \cite{zeghib1999geodesic}. In other words, the local flow of any vector field tangent to $\mathcal{N}$ preserves the induced degenerate Riemannian metric on $F$.

\begin{theorem}\label{Theorem: completeness of lightlike tot geod foliation}
    Let $(M,g,\mathfrak{F})$ be a compact Lorentzian manifold with a codimension $1$ totally geodesic lightlike foliation $\mathfrak{F}$. Then the leaves of $\mathfrak{F}$ are complete with respect to the induced connection if and only if the (null) geodesics tangent to $\mathcal{N}$ are complete.
\end{theorem}

\subsubsection{\textbf{Geodesic equation}}
Let $\dim M = n+2$. 
Near any point $p \in M$, there exist local coordinates $(u, v, x = (x^1 , \ldots , x^n ))$ where the metric has the following form (see \cite[Proposition 2.2]{kundt}): 
$$g=2 du dv + H(u, v,x)du^2+ \sum_{i=1}^{n}W_i(u, v,x) du dx^i + \sum_{i,j} h_{ij}(u,x) dx^i dx^j.$$
We will refer to them as \textit{adapted local coordinates}. 

Set $V:=\partial_v$. 
It is easy to check that $\nabla_{V} V=0$, so the local vector field $V$ is geodesic: its flow parametrizes the null geodesics in $\mathcal{N}$. 
The leaves of $\mathfrak{F}$ are given by the levels of $u$.
For a geodesic tangent to $\mathfrak{F}$, the $u$-coordinate is constant along the geodesic, so we have $\dot{u}=0$. 
And the geodesic equation is the following 
\begin{align*}
    &\ddot{x}^k(t) + \Gamma_{ij}^k(u_0,x) \dot{x}^i \dot{x}^j = 0\\
    &\ddot{v}(t) + \Gamma_{ij}^u(u_0,v,x)  \dot{x}^i \dot{x}^j = 0\\
    &u=u_0
\end{align*}
The fact that the coefficients in $x(t)$ do not depend on $v$ follows from the fact that the foliation $\mathcal{N}$ is transversely Riemannian. 

\subsubsection{\textbf{Proof of Theorem \ref{Theorem: completeness of lightlike tot geod foliation}}} 
Let $E$ be a smooth $n$-dimensional  distribution on $M$, such that for any $p \in M$, $E_p$ is a non-degenerate Riemannian subspace of $T_p M$ tangent to $\mathfrak{F}$. This can be defined by taking an auxiliary Riemannian metric $h$, and setting $E:=T\mathcal{N}^{\perp_{h}} \cap T\mathcal{N}^{\perp_{g}}$. 

For every $p \in M$,  there exists a convex neighborhood $O_p$ of\, $0$\, in $E_p$. The exponential map $\exp_p: E_p \subset T_p M \to M$ restricted to $E_p$ is a diffeomorphism from $O_p$ onto its image in $M$. Define the $n$-submanifold  $S_p := \exp_p O_p$ through $p$.  For every $p \in M$, consider a normal coordinate system $(x^1,\ldots,x^{n})$ on $S_p$. If $\rho = \sum (x^i)^2$, then for $r_p >0$ sufficiently small, $B(p,r_p) = \{q \in S_p, \rho(q) < r_p \}$ is a normal neighborhood of $p$ in $S_p$ diffeomorphic to $B_{n}(0,r_p)$, an open ball of $\R^{n}$ of center $0$ and radius $r_p$. 
Then, we define adapted local coordinates in a neighborhood $U_p$ of $p$: 
\begin{align*}
\phi_p=	(u,v,x=(x^1,\ldots,x^{n})): U_p \to I_p \times J_p \times B_{n}(0,r_p),
\end{align*}
where $I_p$ and $J_p$ are open intervals of $\R$.  
Since $M$ is compact, a continuity argument shows that there exists $r>0$ such that for any $p \in M$, we can take $B(p,r_p)$ to be of uniform radius $r>0$.

\begin{observation}\label{Obs: convex balls} 
$(1)$ Consider adapted local coordinates $(u,v,x)$ on $U$, with the associated coordinate vector fields $Z=\partial_u, V=\partial_v$.	The choice of other local coordinates on $U' \subset U$ gives rise to  coordinate vector fields $Z_1=\partial_{u_1}, V_1=\partial_{v_1}$ such that, when restricted to a leaf of $\mathfrak{F}$, $V_1$ is a rescaling of $V$ by some constant. Indeed, since both $Z$ and $Z_1$ map a leaf of $\mathfrak{F}$ to a leaf of $\mathfrak{F}$, then $Z_1= f(u) Z + W$ for some $\mathfrak{F}$-invariant (local) function $f$ and some (local) vector field $W$ tangent to $\mathfrak{F}$. Since $V_1$ is determined by the equation $g(Z_1,V_1)=1$, the claim follows.\medskip

$(2)$ Given adapted local coordinates on $U$, for any $v$, the subset $\{0\} \times \{v\} \times B_n(0,r)$ is also an open ball of radius $r$ for the induced  Riemannian metric on it. This follows from the fact that the foliation $\mathcal{N}$, when restricted to each leaf $F$ of $\mathfrak{F}$, is transversely Riemannian. More explicitly, the local flow of the vector field $\partial_v$, which is tangent to $\mathcal{N}$ in $U$, preserves the induced degenerate Riemannian metric on $F \cap U$.
\end{observation}

\noindent We can assume  without loss of generality that $r=2$. 
\medskip

\noindent \textbf{Notation:} Denote by $h_{u}$ the  Riemannian metric induced on the slice $\{u\} \times \{v\} \times B_{n}(0,2)$, for some $v$ (it is independent of $v$).

\begin{remark}\label{Remark: x(t)}
Let $\gamma$ be a geodesic tangent to $\mathfrak{F}$. Let $(u,v,x)$ be a system of coordinates around $\gamma(0)$, as above. Then, the $h_u$-norm of $\dot{x}(t)$ is independent of the coordinate system, since it coincides with $g(\dot{\gamma}(t),\dot{\gamma}(t))$. 
And the second order differential equation on $x(t)$ reads
\begin{equation}\label{x(t)}
    \nabla^{h_{u}}_{\dot{x}(t)} \dot{x}(t)=0.
\end{equation}

It is  the geodesic equation associated to the Riemannian metric $h_u$, with $u:=u(\gamma(0))$. 
\end{remark}

\begin{lemma}
Assume that the null geodesics tangent to $\mathcal{N}$ are complete.
Let $\gamma$ be a (maximal) geodesic tangent to $\mathfrak{F}$.  If $g(\gamma'(0),\gamma'(0))=1$, then $\gamma(t)$ is defined on $[0,1]$.  
\end{lemma}
\begin{proof}
Let $p:=\gamma(0)$. Let $F$ be the leaf containing $\gamma$. We denote by $\parallel \cdot \parallel$ the seminorm associated to the degenerate Riemannian metric on $F$. 
Consider the local chart $U_p$, with adapted coordinates 
 $(u,v,x) \in I \times J \times B_n(0,2)$ as above. 
We have $u(t) \equiv 0$.  So $u(t) \in I$, and it is defined as soon as  $x(t)$ and $v(t)$ are defined. By assumption, $\parallel \dot{x}(0) \parallel =1$. 
By Remark \ref{Remark: x(t)},  on its interval of definition, $x(t)$ coincides with the solution $y(t) \subset B_{n}(0,2)$ of the geodesic equation (\ref{x(t)}) associated to the Riemannian metric $h_{0}$, with initial conditions $y(t)=0$ and $\parallel \dot{y}(0) \parallel =1$.  This solution $y(t)$ exists on $[0,1]$, and $\parallel \dot{y}(1) \parallel = \parallel \dot{y}(0) \parallel=1$. However,  $v(t)$ may  leave the interval $J$ before $x(t)$ leaves the ball $B_{n}(0,1)$, meaning that $x(t)$, and consequently $\gamma_{\vert U_{p}}$, is not defined on $[0,1]$. Assume this is the case. 
We will show that $\gamma_{\vert U_{p}}$ extends to a geodesic defined on $[0,1]$. 

For this, let $c(t)$ be the maximal (parameterized) geodesic tangent to $\mathcal{N}$ at $p$.  Let $p_0 := c(t_0)$, and consider a neighborhood $U_{p_0}$ of $p_0$, on which we have adapted local coordinates $\phi_0=(u_0,v_0,x_0): U_{p_0} \to I_0 \times J_0 \times B_n(0,2)$. Since the  null geodesics tangent to $\mathcal{N}$ are complete, the (smooth) slices $S_0(v):=\phi_0^{-1}(I_0 \times \{v\} \times B_n(0,2)) \subset U_{p_0}$, for $v \in J_0$, converge to a submanifold $B_0$ of $M$ which is transversal to $\mathcal{N}$. By Observation  \ref{Obs: convex balls} $(2)$, the intersection of this submanifold with $F$ is an open ball of radius $2$ for the induced Riemannian metric on it. 
Let $p_1 := c(t_1) \in B_0$. Denote by $Z_0, V_0, X^1_0, \ldots, X_0^n$ the coordinate vector fields on $U_{p_0}$. 
These vector fields can be extended  to coordinate vector fields $Z_1, V_1, X^1_1, \ldots, X^n_1$ on some neighborhood $U_{p_1}$ of $p_1$, in such a way that the associated coordinates $(u_1,v_1,x_1) \in I_1 \times J_1 \times B_n(0,2)$ are adapted coordinates satisfying $u_1=u_0, v_1=v_0 , x_1=x_0$ on $U_{p_0} \cap U_{p_1}$. We say that these coordinates are compatible with those on $U_{p_0}$.

One can define this way a  sequence of compatible local charts $U_i:=U_{p_i}$, where $p_i:= c(t_i)$ is an increasing sequence of points on $c(t)$ such that $p_0=p$, $U_i \cap U_j \neq \emptyset$ if and only if $j=i+1$. Denote by $\phi_i=(u_i, v_i,x_i) \in I_i \times J_i \times B_n(0,1)$  the local adapted coordinates on $U_i$. We claim that such a sequence of open sets can be obtained with $t_i \to \infty$. To see this, define the slices $S_i(v) := \phi_i^{-1}(I_i \times \{v_i\} \times B_n(0,2)) \subset U_i$, and denote by $S(p_i)$ the slice containing $p_i$. Suppose that the above sequence does not satisfy $t_i \to \infty$. Since $c(t)$ is complete, the points $p_i$ converge to a point $p \in c(t)$, and the slices $S(p_i)$ (which are leaves of the foliation defined by the levels of $v_i$ in $U_i$) converge to a submanifold $S(p)$ through $p$. The intersection of this submanifold with $F$ is an open ball of radius $2$ for the induced Riemannian metric on it. 
Define adapted coordinates $(u,v,x)$ in a neighborhood $U$ of $p$, such that the slice through $p$ in $U$ given by the level of $v$ at $p$  coincides with $S(p)$. We call this `a good overlapping property'. Set $V=\partial_v$.
There exists some $U_n$ such that $U_n \cap U \neq 0$. By Observation \ref{Obs: convex balls} $(1)$, the restriction of $V$ to a leaf of $\mathfrak{F}$ is a rescaling of the restriction of $V_n$ to that leaf by some constant. Due to this and to the good overlapping property, the slices defined by the levels of $v$ in $U \cap U_n$ coincide with those defined by the levels of $v_n$. 
Thus, we can define new adapted coordinates on $U$ that are compatible with those on $U_n$, thereby extending the previous sequence beyond $p$.

By construction, the geodesic equations associated to two overlapping open sets $U_i$ and $U_{i+1}$ coincide on $U_i \cap U_{i+1}$, giving rise to a single geodesic equation:
\begin{align*}
    &\ddot{x}^k(t) + \Gamma_{ij}^k(0,x)\, \dot{x}^i \dot{x}^j = 0\\
    &\ddot{v}(t) + \Gamma_{ij}^{u}(0,v,x)\,  \dot{x}^i \dot{x}^j = 0\\
    &u=0
\end{align*}
where the coefficients are, this time, defined on $\{0\} \times \R \times B$, with the closure of $B_n(0,1)$ contained in $B$.
The fact that $v$ is defined on $\R$ follows from the fact that the geodesics tangent to $\mathcal{N}$ are complete. 

Thus, to the geodesic $\gamma_{\vert \bigcup_{i} U_i}$ corresponds a maximal geodesic $\widetilde{\gamma}$ in $\{0\} \times \R \times B$, defined on the same interval as $\gamma_{\vert \bigcup_{i} U_i}$, and satisfying the geodesic equation above.  For this geodesic $\widetilde{\gamma}$, it is clear that $x(t)$ is defined on $[0,1]$, and $v(t)$ is defined on the same interval as $x(t)$. Consequently, $\widetilde{\gamma}$, and therefore $\gamma$, is defined on $[0,1]$. 
\end{proof}

\begin{proof}[Proof of Theorem $\ref{Theorem: completeness of lightlike tot geod foliation}$]
This follows from the lemma above, and the fact that the $g$-norm of $\gamma'(t)$ is preserved on $\gamma$. 
\end{proof}
\subsection{Applications} As a consequence of Theorem \ref{Theorem: completeness of lightlike tot geod foliation}, we can state a leafwise completeness result in the  following situations
\begin{corollary}[\textbf{Leafwise completeness in presence of a parallel null line field}] 
 Let $(M,l)$ be a compact Lorentzian manifold with a parallel null line field $l$. Then the leaves of the foliation tangent to the distribution $l^\perp$ are complete if and only if the geodesics tangent to $l$ are complete.   
\end{corollary}

\begin{corollary}[\textbf{Leafwise completeness in presence of a null Killing field}]\label{Cor: Killing complete} 
 Let $(M,V)$ be a compact Lorentzian manifold with a null Killing field $V$, such that the distribution orthogonal to $V$ is integrable and defines a foliation $\mathfrak{F}$. Then the leaves of $\mathfrak{F}$ are complete.   
\end{corollary}

\section{On completeness of compact Lorentzian manifolds with a null Killing field}\label{section: On completeness of compact Lorentzian manifolds with a null Killing field} 
Let $(M,g,V)$ be a compact Lorentzian manifold with a null Killing field $V$. Recall that for any geodesic $\gamma$, we have $g(\dot{\gamma},V)=C$ constant along $\gamma$, and $C$ is called the Clairaut's constant. In this section, we deal with the global completeness problem, i.e. completeness of the whole manifold $M$. We have the following result from the previous section 
\begin{theorem}\label{Theorem: Killing complete} 
 Let $(M,V)$ be a compact Lorentzian manifold with a null Killing field $V$, such that the distribution orthogonal to $V$ is integrable and defines a foliation $\mathfrak{F}$. Then the leaves of $\mathfrak{F}$ are complete.   
\end{theorem}

We construct in this section an incomplete compact Lorentzian manifold with a null Killing field, which is locally homogeneous. We begin with some preliminaries on left invariant metrics on Lie groups. 

\subsection{Left invariant metrics on Lie groups}

Let $G$ be a connected Lie group endowed with a left invariant pseudo-Riemannian metric, $i.e$ its Lie algebra $\g$ is endowed with a non-degenerate inner product $q$. Given any curve $\gamma$ in $G$, we can consider the associated curve $y(t):=dL_{\gamma(t)}^{-1}(\Dot{\gamma}(t)) \in \g$. By Euler-Arnold theorem, the geodesics of $G$ are in one-to-one correspondence with the integral curves of the (Euler-Arnold) quadratic-homogeneous vector field on $\g$:
\begin{equation}\label{E-A}
    y\in \g\mapsto \dot{y}=\ad_y^*y\in \g
\end{equation}
\subsection*{Null Killing fields}  Let $\mathsf{v} \in \g$ such that $\ad_\mathsf{v}$ is skew symmetric with respect to $q$. Then the left invariant metric on $G$ is also invariant by the right multiplication by $\exp(t \mathsf{v})$. So the left invariant vector field $V$ generated by $\mathsf{v}$ is Killing, and it is everywhere null. Moreover, it acts isometrically (on the right) on $\Gamma \backslash G$, where $\Gamma \subset G$ is a cocompact lattice, defining a null Killing field on the compact quotient. 

\subsection*{Completeness of  geodesics orthogonal to $V$} It follows from the observation below that when $\mathsf{v}$ is null and $\ad_\mathsf{v}$ is skew symmetric, the left invariant distribution generated by the linear hyperplane $\mathsf{v}^\perp$ is lightlike and totally geodesic, and the geodesics tangent to it are exactly those orthogonal to the null Killing field $V$. We prove in Proposition \ref{Proposition: (LIM) completeness of  geodesics orthogonal to V} that those geodesics are complete (without assuming any compactness assumption).
\begin{observation}
When $\ad_\mathsf{v}$ is skew symmetric, the vector field $\dot{y}$ in $(\ref{E-A})$ is everywhere tangent to the affine hyperplanes parallel to $\mathsf{v}^\perp$.   
\end{observation}
\begin{proof}
Indeed, one has $g(\ad^*_y y, \mathsf{v})=g(\ad_y \mathsf{v},y)=-g(\ad_\mathsf{v} y,y)=0$.   
\end{proof} 
\begin{proposition}\label{Proposition: (LIM) completeness of  geodesics orthogonal to V}
 Let $G$ be a Lie group with a left invariant metric such that there is $\mathsf{v} \in \g$ with $q(\mathsf{v},\mathsf{v})=0$ and  $\ad_\mathsf{v}$ skew symmetric. Then the integral curves of the quadratic-homogeneous vector field $(\ref{E-A})$ tangent to the linear hyperplane $\mathsf{v}^\perp$ are complete. Equivalently, geodesics orthogonal to the (left invariant) null Killing field generated by $\mathsf{v}$ are complete.  
\end{proposition}
\begin{proof}
Let $(\mathsf{v},\mathsf{x}=(\mathsf{x_1},\ldots,\mathsf{x_n}),\mathsf{u})$ be a basis of $\g$ such that $(\mathsf{v}, \mathsf{x})$ generates $\mathsf{v}^\perp$. Let $y(t)$ be an integral curve of $\dot{y}$ tangent to $\mathsf{v}^\perp$, and write $y(t)=(v(t),x(t),u(t))$ in the above basis. The hyperplane $\mathsf{v}^\perp$ is lightlike with a transverse Riemannian quadratic form. 
Since the geodesic vector field is tangent to the levels of the quadratic form $q$, then on $\mathsf{v}^\perp$, it is tangent to levels of the form $\mathbb{S}^n \times \R$, where $\R$ is the direction of $\mathsf{v}$ and $\dim \mathsf{v}^\perp=n+1$. 
  In particular, the $x$-component is bounded, hence complete. On the other hand, one can check that $\ad^*_{\mathsf{v}} \mathsf{v} =0$; as a consequence, Equation (\ref{E-A}) on $y(t)$ yields an equation on $\dot{v}$ which is linear in $v$, hence $v(t)$ is also complete. 
\end{proof}
\subsection*{Dynamics of $\dot{y}$} We saw that the integral curves of $\dot{y}$ on the linear hyperplane $\mathsf{v}^\perp$ correspond to geodesics orthogonal to $V$, which are complete. On the other hand, the integral curves of $\dot{y}$ on the affine translates of $\mathsf{v}^\perp$ correspond to the geodesics with Clairaut's constant $C \neq 0$. After affine reparametrization, we can suppose $C=\pm 1$, so that the dynamics of the vector field $\dot{y}$ reduces to the ones on the two hyperplanes $\mathsf{v}^\perp \pm \mathsf{u}$, where $\mathsf{u} \in \g$ is such that $q(\mathsf{u},\mathsf{v})=1$. The incompleteness obtained in the next paragraph occurs on those hyperplanes. 
\subsection{Incomplete examples with a null Killing field}
\begin{example}[\textbf{An incomplete left invariant Lorentzian metric with a null Killing field}]\label{Example: LIM incomplete with null Killing}
Consider the Lie algebra $\g=\aff(\R)\oplus \R$. Let $(T,X,V)$ be a basis such that $[T,X]=X$. Let $q$ be the Lorentzian quadratic form on $\g$ given by $$q=\begin{pmatrix}

1 & 0 & 0  \\
0 & 0 & 1  \\
0 & 1 & 0 
\end{pmatrix},$$
in the basis $(T,X,V)$.
Since $V$ is central, $\ad_V$ is skew-symmetric with respect to $q$. Therefore, the left invariant vector field generated by $V$ is null and Killing.
The Euler-Arnold vector field in this case is $$w=\begin{pmatrix}
t \\
x  \\
v 
\end{pmatrix} \in \g\mapsto \ad^*_ww= \begin{pmatrix}

-xv \\
0  \\
tv 
\end{pmatrix}.$$
So, we have the system of equations
\begin{equation*}
    \begin{cases}
     \Dot{t}=-xv\\
     \Dot{x}=0\\
      \Dot{v}=tv
    \end{cases}
\end{equation*} Moreover, $q(\omega,\omega)=t^2+2xv= C$, $C \in \R$. For $x=1$, we get that $\Ddot{t}=-\Dot{v}=-tv=t\Dot{t}=(\frac{1}{2}t^2)'$. Hence $\Dot{t}=\frac{1}{2}(t^2 - C)$.  So $w(s)=(\frac{-2}{s}, \frac{2}{s^2}, 1)^T$ is the velocity curve of an incomplete lightlike geodesic on the Lie group $G=\Aff(\R)\times\R$ endowed with the left invariant metric coming from $q$. 
\end{example}
\begin{example}[\textbf{An incomplete compact Lorentzian manifold with a null Killing field}]
Let $\g=\sol\oplus \R$. Let $(T,X,Y,V)$ be a basis such that $[T,X]=X, [T,Y]=-Y$, and where $V$ generates the center. Consider the Lorentzian quadratic form $q$ given by $$q=\begin{pmatrix}
1 & 0 & 0 & 0 \\
0 &0 & 0 & 1\\
0 & 0 & 1 & 0\\
0 & 1 & 0 & 0
\end{pmatrix}$$
in the basis $(T,X,Y,V)$. As in the previous example, the left invariant vector field generated by $V$ is null and Killing.
In this case, the Arnold-Euler vector field is 
$$w=\begin{pmatrix}
t \\
x  \\
y \\
v
\end{pmatrix}\in \g\mapsto \ad^*_ww= \begin{pmatrix}

-xv+y^2 \\
0  \\
ty \\
tv
\end{pmatrix}.$$
In other words, we get the following system
\begin{equation*}
    \begin{cases}
    \Dot{t}=-xv+y^2\\
    \Dot{x}=0\\
    \Dot{y}=ty\\
    \Dot{v}=tv 
    \end{cases}    
\end{equation*}
We see that the geodesic field is tangent to the hyperplane $y=0$, which is then totally geodesic. On the other hand, this hyperplane is the subalgebra $\aff(\R) \oplus \R$ with the (incomplete) Lorentzian scalar product given in Example \ref{Example: LIM incomplete with null Killing} above.  Hence incompleteness. We obtain a locally homogeneous compact example by taking the quotient of $\SOL\times \R$ by a cocompact lattice. 
\end{example}
\begin{remark} In the previous example, the incomplete geodesics are timelike and lightlike.
\end{remark}

\section{Compact Lorentzian $3$-manifolds with an equicontinuous null Killing field}\label{section: Dimension 3: an example and a problem}
Let $(M,V)$ be a compact Lorentzian manifold of dimension $3$, with $V$ a null Killing field. In this section, we study the geodesic completeness of $M$. As pointed out in the introduction, the situation is different,  depending on the dynamics of $V$. The non-equicontinuous case being already known in dimension $3$, we focus on the equicontinuous case. 
In this case, the closure of the flow of $V$ in the isometry group is a torus. Denote it by $G$. \\

\paragraph{\textbf{Reading plan}} We begin by a study of the connected component of the isometry group in Paragraph \ref{subsection: On the isometry group}. We will see in particular that when it has dimension $\geq 3$, the manifold is flat, hence complete. An incomplete example is constructed in Paragraph \ref{subsubsection: dim G=1}, in the case where the flow of $V$ is periodic. The reader can proceed directly to Paragraph \ref{subsection: Existence of a global frame field} for this example (which is independent of Paragraph \ref{subsection: On the isometry group}). Completeness results are also given in Paragraphs \ref{subsubsection: dim G=2} and \ref{subsubsection: dim G=3}, in special situations where the action of $V$ is not periodic.

\subsection{Isometry groups}\label{subsection: On the isometry group}  The aim of this section is to prove the following result

\begin{proposition}\label{Prop: structure of G}
Let $(M,V)$ be a connected compact Lorentzian $3$-manifold with an equicontinuous null Killing field $V$, with closure $G$ in $\Isom(M)$. 
\begin{enumerate}
    \item If $\dim G \geq 3$, then the identity component of the isometry group is isomorphic to $\T^3$, in which the flow of $V$ is dense, and $M$ is a flat $3$-torus.
    \item If $\dim G =2$, then either the connected component of the isometry group is $G \simeq \T^2$, or, $M$ is flat and $\Isom^o(M)=\T^3$.
\end{enumerate}
\end{proposition}

We start with the following two lemmas

\begin{lemma}\label{Lemma: 3D flat}
Let $M$ be a connected compact Lorentzian $3$-manifold. If $\Isom(M)$ contains a $3$-dimensional connected solvable subgroup $H$, then $M$ is $H$-homogeneous. If $H$ is abelian, then $M$ is flat, and  the identity component of the isometry group is isomorphic to $\T^3$.
\end{lemma}
\begin{proof}
Assume first that $H$ is compact, i.e. $H=\T^3$. Let $K_1, K_2, K_3$ be three transversal Killing fields generated by the action of $H$. Let $C_i$ be the set of zeros of $K_i$, and let $C=\bigcup_{i=1}^3 C_i$. Each of the $C_i$ has empty interior (this comes from the fact that the action of the isometry group is faithful). Consequently, $C$ is a closed subset of $M$ with empty interior. So let $x \in M \smallsetminus C$. Then the $H$-orbit of $x$ is an open and closed subset of $M$, hence equal to $M$. This implies that the $H$-action is locally free and transitive. 
Now, assume that $H$ is non-compact. By the construction above, there exists a point $x_0 \in M$ with trivial local isotropy. Hence, its orbit is a homogeneous open subset of $M$ (of finite volume). Since the action of $H$ on $U_0$ is locally free, this orbit is  given by $U_0=H / \Gamma_0$, where the isotropy group $\Gamma_0$ is a discrete subgroup of $H$.  
Now, since $H / \Gamma_0$ has finite volume (with respect to the volume element induced by the Lorentzian metric of $U_0$), and $H$ is solvable, the subgroup $\Gamma_0$ must be a cocompact lattice in $H$. Therefore, $U_0$ is compact, making it a closed subset of $M$. Finally, since $M$ is connected, we conclude that $U_0=M$. This proves the first part of the lemma. Now, if $H$ is abelian, then  $M$ is clearly flat. In this case, $M$ is flat and compact, hence complete. So its full isometry group is given by $\Gamma\backslash N(\Gamma)$, where $N(\Gamma)$ is the normalizer of $\Gamma:=\pi_1(M)$ in the Poincar\'e group $\mathsf{Poi}^{2,1}=\O(2,1)\ltimes \R^{2+1}$. As $H \subset \Isom(M)$, the normalizer $N(\Gamma)$ contains a $3$-dimensional connected abelian group, which must be the translation group $\R^3$ by Observation \ref{Observation: maximal abelian subgroup of Poi_1,2} below. Hence, the subgroup of translations $\R^3$ centralizes $\Gamma$, and this implies that $\Gamma$ consists of pure translations. 
\end{proof}
\begin{observation}\label{Observation: maximal abelian subgroup of Poi_1,2}
    The maximal dimension of an abelian connected subgroup of the Poincar\'e group $\Poi^{2,1}:=\Isom(\Mink^{2,1})$ is $3$. Moreover, it is identified with the subgroup of pure translations.
\end{observation}

\begin{lemma}\label{Lemma: constant curvature}
A compact Lorentzian $3$-manifold of constant curvature  has an isometry group of dimension $\leq 3$.
\end{lemma}
\begin{proof}
In this case, $M$ is complete by \cite{carriere1989autour, klingler1996completude}, hence can be identified with a compact quotient of $\Mink^{2,1}$ or $\widetilde{\AdS^{3}}$. In both cases $M$ is locally homogeneous. Thus, the stabilizer in $\Isom(M)$ of a point $p \in M$  is isomorphic to the stabilizer of any other point.

We start with the flat case. In this case, $M=\Gamma\backslash \Mink^{2,1}$, where $\Gamma\subset \O(2,1)\ltimes \R^3$ is a discrete subgroup that acts properly, freely, and cocompactly. Assume $\dim \Isom(M) \geq 4$.  We will show that $M$ is homogeneous under the action of the isometry group. Let $I$ denote the isotropy of some point in $M$. If $\dim I=1$, then the orbits of the isometry group are all open, so all closed. Since $M$ is connected, it consists of a single orbit, so homogeneity follows. Now, if $\dim I\geq 2$, then, by \cite[Proposition 2.3]{allout2022homogeneous}, it is, in fact, the whole $\O(2,1)$ up to finite index. On the other hand, the isometry group of $M$ is given by $\Gamma \backslash N(\Gamma)$, where $N(\Gamma)$ is the normalizer of $\Gamma$ in $\O(2,1)\ltimes \R^3$.  In particular, the isotropy group $\O(2,1)$ normalizes $\Gamma$. This is impossible, since such a group $\Gamma$ must be finite. Finally, $M$ is homogeneous as claimed. In the latter case, $M$ is necessarily the quotient of $\R^3$ by a lattice $\Gamma$ of translations. This can be shown by elementary arguments (in dimension $3$), but we refer to \cite[Corollary 4.9]{baues} for this fact. The centralizer of $\Gamma$ is $\R^3$, so the identity component of the isometry group is isomorphic to $\T^3$.

In the anti-de Sitter case, $M$ is a compact quotient $\widetilde{\Gamma} \backslash \widetilde{\AdS^3}$ by a discrete subgroup $\widetilde{\Gamma} \subset \Isom(\widetilde{\AdS^3})$ acting properly, freely and cocompactly. Assume $\dim \Isom(M) \geq 4$. 
Then, the normalizer of $\widetilde{\Gamma}$ in $\Isom(\widetilde{\AdS^3})$ has dimension at least $4$.  
Denote by $Z$ the center of $\Isom(\widetilde{\AdS^3})$, it is isomorphic to $\Z$, and $\Isom(\widetilde{\AdS^3})/Z =\Isom(\AdS^3)$. Denote by $\pi: \Isom(\widetilde{\AdS^3}) \to \Isom(\AdS^3)$ the natural projection. 
A covering of $M$ also has an isometry group of dimension $\geq 4$. So, up to taking a finite cover, we can assume that $\widetilde{\Gamma} \subset \Isom^{\mathsf{o}}(\widetilde{\AdS^3})$. 
Then, by \cite[Theorem 7.2]{kulkarni}, $\widetilde{\Gamma}$ intersects the center $Z$ non-trivially, which guarantees that the quotient group $\Gamma:= \widetilde{\Gamma}/ \widetilde{\Gamma} \cap Z$ is a subgroup of $\Isom(\AdS^3)$ that acts properly, freely and cocompactly on $\AdS^3:= \widetilde{\AdS^3}/Z$. The manifold $M$ is a (finite) cover of $\overline{M}:=\Gamma \backslash \AdS^3 = \widetilde{\Gamma} Z \backslash \widetilde{\AdS^3}$.
By assumption,  $\widetilde{\Gamma}$ has a normalizer, denoted by $C$,  of dimension at least $4$ in $\Isom(\widetilde{\AdS^3})$.  
Now, the fact that the kernel of the projection $\pi$ is discrete implies that $\Gamma$ has a normalizer of dimension at least $4$ in $\Isom(\AdS^3)$, given by the projection of $C$.  
The advantage of considering $\overline{M}$ is that we have a better understanding of its algebraic form. Indeed, by \cite[Theorem 5.2]{kulkarni}, up to conjugacy, there exists a (torsion free) cocompact lattice $\Gamma_0$ in $\PSL_2(\R)$, and a representation $\rho: \Gamma_0 \to \PSL_2(\R)$, such that $\Gamma=\{(x, \rho(x)), x \in \Gamma_0\}$. Let $p_1:\PSL_2(\R) \times \PSL_2(\R) \to \PSL_2(\R)$ be the projection on the first factor. Since $\dim C \geq 4$, $\Gamma_0$  has at least a $1$-dimensional normalizer in $\PSL_2(\R)$, given by $p_1(C)$. 
This normalizer acts on the compact hyperbolic surface $\Gamma_0 \backslash \H^2$ by (at least) a $1$-dimensional group of isometries, contradicting the fact that the isometry group of a compact hyperbolic surface is finite.
\end{proof}

Recall that the foliation tangent to $V^\perp$ is lightlike and totally geodesic. For the proof of Proposition \ref{Prop: structure of G}, we need the following general fact
\begin{fact}[Theorem 11, \cite{zeghib1999geodesic}]\label{Fact: geod lightlike foliation implies contractible}
Let $M$ be a compact Lorentzian $3$-manifold with a $C^0$ lightlike geodesic foliation. Then the universal cover of $M$ is diffeomorphic to $\R^3$.     
\end{fact}
\begin{proof}[Proof of Proposition \ref{Prop: structure of G}]
      $(1)$ By assumption, the closure of the flow of $V$ in $\Isom(M)$ is $G=\T^k$, with $k \geq 3$. The conclusion follows from Lemma \ref{Lemma: 3D flat}.
      
      $(2)$ Let $\G:=\Isom^\mathsf{o}(M)$. First, recall that we have a bound  on the dimension of $\G$, namely, $\dim \G \leq 6$. 
      The proof is based on the classification result on the connected component of the isometry group of a compact Lorentzian manifold \cite{adams1997isometry1,adams1997isometry2,zeghib1998identity}. By \cite[Theorem 1.5]{zeghib1998identity}, the Lie algebra of $\G$ decomposes as $\widehat{g}=\R^k \oplus \mathcal{K} \oplus \mathcal{H}$, where $\mathcal{K}$ is the Lie algebra of a compact semisimple Lie group and $\mathcal{H}$ is isomorphic to either  $\sl_2(\R), \heis$, or $\osc$ the Lie algebra of an  oscillator group. In our case, $\R^2 \subset \widehat{g}$. We need to consider the following two cases:
      \begin{itemize}
          \item[I.] If $\dim(\widehat{g})=3$, then $\widehat{g}$ is isomorphic to either $\R^3$ or $\heis_3$. Indeed, the $\sl_2(\R)$ and $\so(3)$ cases are excluded, since their maximal abelian subgroup has dimension $1$. Thus, $\widehat{G}$ is solvable, and by Lemma \ref{Lemma: 3D flat}, $M=\widehat{G}/ \Gamma$, where $\Gamma$ is a cocompact lattice in $\widehat{G}$. Since $\G$ contains $\T^2$, the  normalizer of $\Gamma$ in $\G$ has dimension $\geq 2$. Therefore, $\widehat{g} = \R^3$. In this case, $M$ is flat by Lemma \ref{Lemma: 3D flat}.
          \item[II.] If $\dim(\widehat{g})\geq 4$ then $M$ is locally homogeneous by \cite[Theorem 4.1]{frances}. Let $I$ be the isotropy group of some point. The case where $\dim I=2$ is not possible by \cite[Proposition 2.3]{allout2022homogeneous}. 
          We are left with two situations. If $\dim I \geq 3$, then $M$ has constant curvature. So $M$ is a compact constant curvature Lorentzian $3$-manifold with an isometry group of dimension $\geq 4$. This is not possible by Lemma \ref{Lemma: constant curvature}. If $\dim I=1$, the orbits of $\widehat{G}$ in $M$ are all open, and therefore also closed. Since $M$ is connected, it consists of a single orbit, meaning $M$ is $\widehat{G}$-homogeneous. Consequently, $M$ can be identified with $\G/I$, where $I$ is the isotropy of some point. In particular, $\dim \widehat{g}= 4$.
          In this case, the possibilities for $\widehat{g}$ are $\R^4$, $\R \oplus \sl_2(\R)$, $\R \oplus \heis_3$, $\R \oplus \so(3)$, and $\osc$. By Lemma \ref{Lemma: 3D flat}, we can rule out all possibilities except for $\R \oplus \so(3)$.
In the latter case, $\G$ is isomorphic to $\R\times K$ or $\mathbb{S}^1\times K$, where $K$ is finitely covered by $\mathbb{S}^3$. In both cases, $M$ has a non-contractible universal cover, diffeomorphic to either $\mathbb{S}^3$ or $\R\times \mathbb{S}^2$. This is not possible by Fact \ref{Fact: geod lightlike foliation implies contractible}.
      \end{itemize}

Finally, either $\dim \G=3$, and $M$ is flat. Or $\dim \G=2$, in which case $\G=G$. This ends the proof. 
\end{proof}
We refer to Paragraph \ref{subsubsection: dim G=2} for examples where the identity component of the isometry group is a $2$-torus.

\subsection{Existence of a $V$-invariant frame field}\label{subsection: Existence of a global frame field}
\begin{lemma}\label{Lemma: 3D construction of null frame}
Let $(M,g,V)$ be an orientable compact Lorentzian manifold of dimension $3$, with $V$ an equicontinuous null Killing field. Then, there exists a null frame field $(V,X,Y)$ on $M$ which is invariant by the flow of $V$, and such that the metric in the basis $(V,Y,X)$ is given by 
$
\begin{pmatrix}
0 & 1 & 0  \\
1 & 0 & 0  \\
0 & 0 & 1 \\
\end{pmatrix}.$
Conversely, given a $V$-invariant frame field $(V,X,Y)$ in which the metric has the above form, $V$ is a null Killing field for this metric. 
\end{lemma}
\begin{proof}
The flow of $V$ preserves a Riemannian metric $h_0$ on $M$. Let $P_0$ (resp. $P$) be the $2$-dimensional distribution $h_0$-orthogonal (resp. $g$-orthogonal) to $V$. Then $P_0 \cap P$ is a line bundle over $M$. Up to double cover, we can define a vector field $X$ on $M$ tangent to the line bundle, and satisfying $g(X,X)=1$.  Let $Y$ be the null vector field on $M$ which is $g$-orthogonal to $X$ and satisfies $g(V,Y)=1$. Since the flow of $V$ is isometric for both $g$ and $h_0$, these vector fields are invariant by the flow of $V$, i.e. $[V,X]=[V,Y]=0$. For the converse, one proves that $\nabla V$ is skew-symmetric for $g$, using that the scalar products $g(U,W)$ are $V$-invariant for $U, W \in \{V,X,Y\}$.    
\end{proof}

In particular, the foliation tangent to $V$ is trasnversally parallelizable. The action of $V$ is assumed to be equicontinuous. So the closure of the leaves of $V$ are the fibers of a fibration defined by the (proper) action of $G$, the closure of the flow of $V$ in the isometry group. When the action of $V$ is free (note that the action also preserves the orientation of $M$), the base space of this fibration is a compact, orientable manifold. Its Euler characteristic is zero, as the $V$-action preserves non-vanishing vector fields transverse to $V$, which project to well-defined vector fields on the base space.  Therefore, the base is  diffeomorphic to a torus (of dimension $\leq 2$), which we denote by $\Sigma$. 

\begin{lemma}\label{Lemma 7.3}
 Let $(V,X,Y)$ be a null frame field as above. Then there exist $V$-invariant functions $f, h, \mu \in C^{\infty}(M,\R)$ such that
 \begin{itemize}
     \item  $\nabla_V V=0$, $\nabla_X V= -\frac{1}{2}f V$, $\nabla_Y V = \frac{1}{2}f X$, 
     \item $[X,Y]=fY + h X + \mu V$.
 \end{itemize}
In particular, $V$ is parallel if and only if $f=0$. Moreover, since $f$, $h$ and $\mu$ are constant along the orbits of $V$, they project to $\Sigma$. 
\end{lemma}
\begin{proof}
 The first point is a straightforward computation, using that $V$ is Killing and all the scalar products involving $V, X,$ and $Y$ are constant. The fact that $f$, $h$ and $\mu$ are $V$-invariant follows from the fact that the frame is $V$-invariant, hence $[X,Y]$ is also $V$-invariant.  
\end{proof}
\subsection{Geodesic completeness} We begin now the investigation of the global completeness.   
\begin{observation}
    In dimension $3$,  the geodesics orthogonal to $V$ are complete. Indeed, the distribution orthogonal to $V$ is integrable, hence complete by Corollary $\ref{Cor: Killing complete}$.
\end{observation}

\subsubsection{\textbf{Case of $\dim G\geq3$}}\label{subsubsection: dim G=3} 
By Proposition \ref{Prop: structure of G}, this case corresponds to a null Killing field $V$ having a dense orbit in $M$ (in which case, all orbits of $V$ are dense).
As shown in Proposition \ref{Prop: structure of G}, $M$ is, in this case, a flat $3$-torus, which is therefore complete, with $\dim G = 3$.

\subsubsection{\textbf{Case of $\dim G=1$}}\label{subsubsection: dim G=1}

In this paragraph, we assume that $\dim G=1$, i.e. the flow of $V$ is periodic. Then $M$ is a Seifert fibration. For simplicity, we will assume that the $\mathbb{S}^1$-action of $V$ is free. Then $V$ is tangent to the fibers of a fiber bundle $\pi: M \to \Sigma$, where $\Sigma$ is diffeomorphic to a $2$-torus. 
\subsubsection*{\textbf{Getting started}}
The following lemma is a converse to Lemma \ref{Lemma 7.3} above. 
\begin{lemma}\label{Lemma: Lemma 7.5}
Let $\pi: M \to \Sigma$ be an $\mathbb{S}^1$-bundle over a $2$-torus $\Sigma$, given by an $\mathbb{S}^1$ free action. Denote by $V$ the vector field on $M$ generated by the $\mathbb{S}^1$-action (it is tangent to the fibers). Let $(X^*, Y^*)$ be two vector fields on $\Sigma$ such that $[X^*, Y^*]=f^* Y^* + h^* X^*$, with $f^*, h^* \in C^{\infty}(\Sigma, \R)$. There exist vector fields $X,Y$ on $M$ together with a Lorentzian metric $g$ on $M$, such that 
\begin{itemize}
    \item $(V,X,Y)$ is a $V$-invariant null frame field on $M$, and $V$ is a null Killing field for $g$,
    \item $[X, Y]=f Y + h X + \mu V$, where $f, h, \mu \in C^{\infty}(M, \R)$ are $V$-invariant functions on $M$, such that $f$ and $h$ project to $f^*$ and $ h^*$ respectively.
\end{itemize}
This metric is not unique.  
\end{lemma}
\begin{proof}
Choose a Riemannian metric $h_0$ on $M$ preserved by the flow of $V$. It defines a horizontal distribution on $M$ by taking the plane distribution orthogonal to $V$. Then $X$ and $Y$ are obtained by taking the horizontal lifts of $X^*$ and $Y^*$ respectively. By definition, $X$ and $Y$ project to $\Sigma$, so the coefficients $f, h, \mu$ are $V$-invariant. Define $g$ so that $(V,X,Y)$ is a null frame as in Lemma \ref{Lemma: 3D construction of null frame}. The fact that $V$ is a Killing field follows from Lemma \ref{Lemma: 3D construction of null frame}.      
\end{proof}

\begin{remark}\label{Remark: trivial bundle}
The distribution $\Span(X,Y)$ is integrable if and only if $\mu=0$.
 In the previous lemma, one may  consider a trivial fiber bundle $\Sigma \times \mathbb{S}^1$. In this case, since we have a global section, one can lift $X^*$ and $Y^*$ so that the distribution generated by $X$ and $Y$ is tangent to the sections, hence integrable.  
\end{remark}
\subsubsection*{\textbf{Geodesic equation}} Let $\gamma$ be a geodesic of $M$ and write 
$$\dot{\gamma}(t)=a(t) V + b(t) X + c(t) Y$$ 
in the global framing. Let $\beta=g(\dot{\gamma}, \dot{\gamma})$ and $\alpha= g(\dot{\gamma},V)$. Both $\beta$ and $\alpha$ are constant, and we have $c(t)=\alpha$ for any $t$. 
\begin{observation}
When $\alpha=0$, $\gamma$ is contained in a leaf of the $\mathfrak{F}$-foliation, and is then complete.    
\end{observation}
In the sequel, we consider geodesics transverse to the $\mathfrak{F}$-foliation, i.e. for which $\alpha \neq 0$. 
\begin{lemma}[\textbf{Computation of Christoffel symbols}]
The only non-null Christoffel symbols are 
\begin{itemize}
    \item $\Gamma_{VY}^X = \Gamma_{YV}^X =-\Gamma_{VX}^Y=-\Gamma_{XV}^Y= \Gamma_{XY}^V=\frac{1}{2}f$ 
    \item $\Gamma_{XY}^X =-\Gamma_{XX}^Y =h$ 
    \item $\Gamma_{YY}^X =-\Gamma_{YX}^Y =\mu$ 
\end{itemize}    
\end{lemma}
\begin{proof}
Straightforward computation, using Lemma \ref{Lemma 7.3}.    
\end{proof}
So the geodesic equation when $\alpha \ne 0$ reads
\begin{align*}
    \dot{b}&= \frac{1}{2}f(\gamma(t)) b^2 - \alpha h(\gamma(t)) b  - \Big{(}\frac{\beta}{2} f(\gamma(t))+\alpha^2 \mu(\gamma(t))\Big{)} \\
    \dot{a}&=-\frac{f(\gamma(t))}{2 \alpha} b^3 + h(\gamma(t))b^2+ \Big{(}\mu \alpha + \frac{\beta}{2 \alpha} f(\gamma(t))\Big{)}b\\
    \dot{c}&=0
\end{align*}
\paragraph{\textbf{Reduction: Lorentzian metric associated to a triple $(\Sigma, X^*,Y^*)$}} 
By Lemma \ref{Lemma: Lemma 7.5} and Remark \ref{Remark: trivial bundle}, one can associate to a triple $(\Sigma, X^*,Y^*)$ such that $[X^*,Y^*]=f^* Y^* + h^* X^*$, a Lorentzian metric on the trivial bundle $\Sigma \times \mathbb{S}^1$ for which the lifts $X, Y$ of $X^*$ and $Y^*$ to $M$ define an integrable distribution. And one has $[X,Y]=fY + h X$, where $f$ and $h$ are fiber-constant functions obtained from $f^*$ and $h^*$ respectively. 
\begin{observation}
 The orbits of $Y$ are totally geodesic. Indeed, we have $g(\nabla_Y Y, Y)=0$ and $g(\nabla_Y Y, V)=-g(Y, \nabla_Y V)=0$.  And under the above reduction, i.e. integrability of the distribution\; $\Span(X,Y)$, we have more: the flow of  $Y$ is geodesic. This follows from $g(\nabla_Y Y, X)=\Gamma_{YY}^X=\mu=0$, which yields $\nabla_YY=0$. In this case, the orbits of $Y$ are complete null geodesics.    
\end{observation}

One asks: 
\begin{question}\label{Question: Dim 3}
    What restrictions on $f^*$ and $h^*$ in order to obtain an incomplete example?
\end{question}

With the above reduction, the equation on $b(t)$ for a null geodesic reads
\begin{align}
\dot{b}=\frac{1}{2} f(\gamma(t)) b^2 -\alpha h(\gamma(t))b.    
\end{align}
\begin{observation}\label{Observation: b(t_0)=0 implies b=0}
 If $b(t_0)=0$ for some $t_0$, then $b$ is identically zero. In this case, $\gamma$ is a complete null geodesic everywhere tangent to $\Span(Y,V)$.   
\end{observation}

\subsubsection*{\textbf{Construction of an incomplete example}} In what follows, the vector fields $X^*$ and $Y^*$ are lifted to $\widetilde{\Sigma}=\R^2$, keeping the same notations. So consider on $\R^2$ the two following linearly independent vector fields, written in the $(x,y)$-coordinates of the plane:
\begin{align*}
    X^*&= \cos(x) \partial_x - \sin(x) \partial_y\\
    Y^*&= \sin(x) \partial_x + \cos(x) \partial_y
\end{align*}
By definition, they project to vector fields on a torus $\Sigma=\T^2$. We have $$[X^*,Y^*]=\partial_x=\cos(x) X^* + \sin(x) Y^*.$$ 

Consider the Lorentzian metric $(M, g, V)$ (with $V$ a null Killing field) on the trivial bundle $M=\Sigma \times \mathbb{S}^1$, associated to the triple $(\Sigma, X^*, Y^*)$. The two functions $f, h \in C^{\infty}(\Sigma, \R)$ are given by $f(x,y)= \sin(x)$ and $h(x,y)=\cos(x)$.  

The foliations determined by $X^*$ and $Y^*$ on $\widetilde{\Sigma}$ are illustrated in Figure 1. 
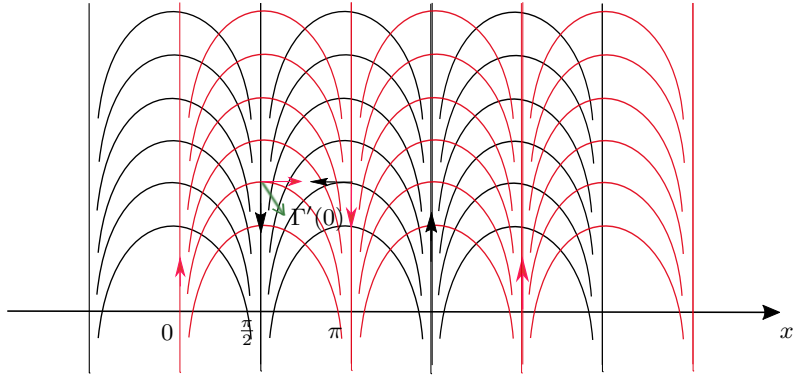
\begin{figure}[h!]
	\labellist 
	\small\hair 2pt 
 	\pinlabel {$\text{$\Gamma'(0)$}$} at 370 180
    \pinlabel {$\text{$0$}$} at 200 50
    \pinlabel {$\text{$\frac{\pi}{2}$}$} at 290 50
    \pinlabel {$\text{$\pi$}$} at 390 50
    \pinlabel {$\text{$x$}$} at 900 50
	\endlabellist 
	\centering 
	\includegraphics[scale=0.33]{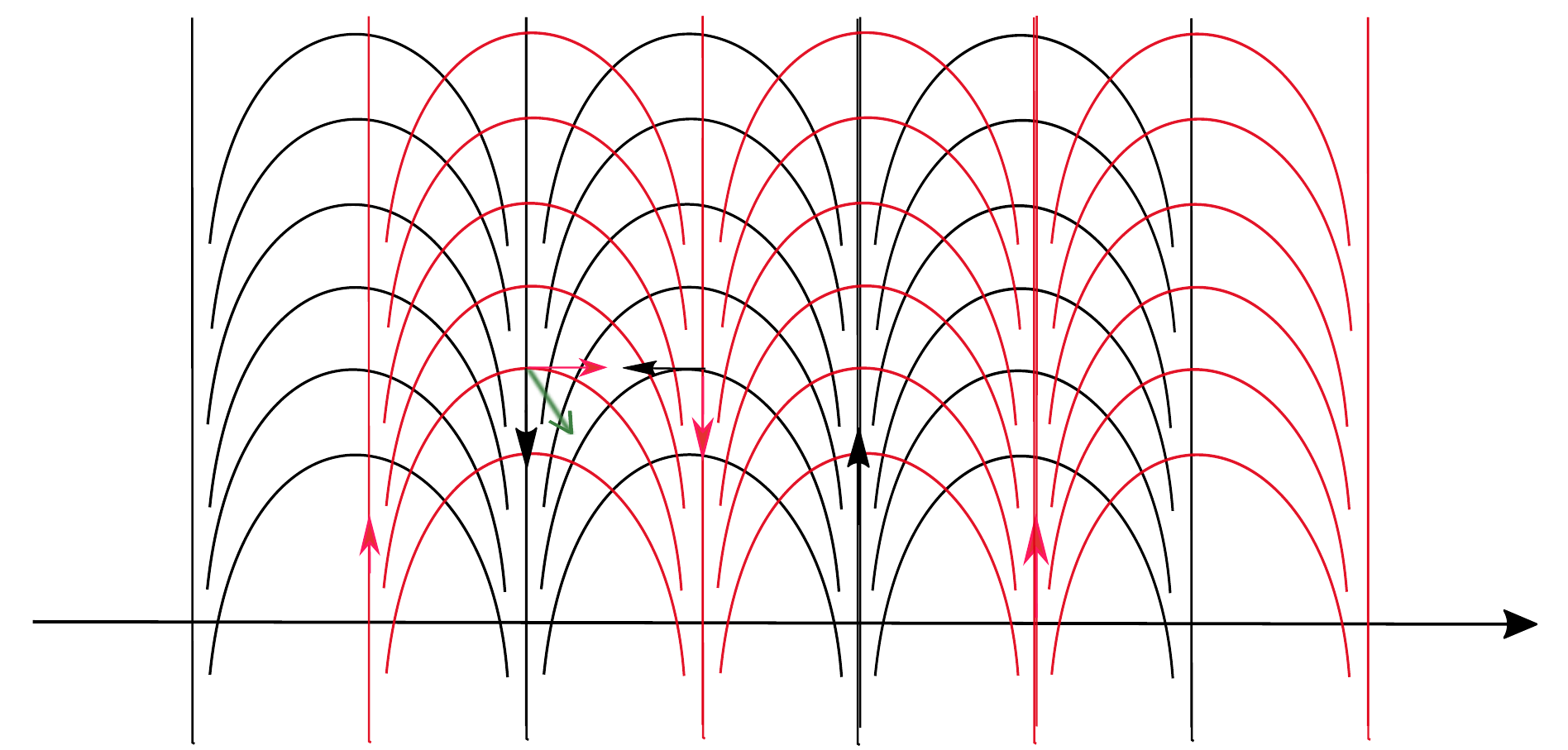} \caption{The foliation tangent to $X^*$ (resp. to $Y^*$) is represented in black (resp. red)}
\end{figure}

They are made of Reeb components. This choice will allow Lemma \ref{Lemma: geodesic behavior} below, which is the key tool to incompleteness phenomena obtained in Corollary \ref{Corollary: incomplete}. 

Let $\gamma$ be a null geodesic of $M$ defined on $I=[0,T[$, $T \in \overline{\R}$, such that $\alpha=1$, and $b(0) >0$. Denote by $\Gamma(t)=b(t) X^* + Y^*$  the projection of $\gamma$ to $\Sigma$ (lifted to $\R^2$). We introduce here some vocabulary we will use in the proof of Lemma \ref{Lemma: geodesic behavior}; it is introduced to facilitate proof writing.\\

\noindent\textbf{Vocabulary:}
\begin{itemize}
    \item We call a \textbf{band} of $X^*$ (resp. of $Y^*$) any single Reeb component of $X^*$ (resp. of $Y^*$), with its boundary leaves.
    \item We call \textbf{half a band} the intersection of a band of $X^*$ and a band of $Y^*$.
    \item We say that $\Gamma$ \textbf{crosses a band} when there exist two distinct $t_0, t_1 \in I$ such that $\Gamma(t_0)$ and $\Gamma(t_1)$ are on distinct boundary leaves of the band. The same can be said in the case of a half band.
    \item We say that $\Gamma$ \textbf{makes a turnaround} when there exist two distinct $t_0, t_1 \in I$ such that $\Gamma(t_0)$ and $\Gamma(t_1)$ belong to the boundary leaf of some band. 
\end{itemize}

\noindent The boundary leaves of the bands coincide with the $x$-axis and its $\frac{k\pi}{2}$-translates, for $k \in \Z$. 

\begin{lemma}\label{Lemma: geodesic behavior}
Let $\gamma$ be a null geodesic of $M$ defined on $I=[0,T[$, $T \in \R \cup \{+\infty\}$, such that $\alpha=1$, and $b(0) >0$. Then
\begin{enumerate}
    \item $\Gamma(t)$ does not cross any band. 
    \item $\Gamma(t)$ does not make turnarounds.
\end{enumerate}
As a consequence, for $t$ close enough to $T$, $\Gamma(t)$ remains in half a band. 
\end{lemma}
\begin{proof} 
\textbf{(1)} $\Gamma(t)$ does not cross any band. Indeed, if $\gamma$ crosses a band of $X$, then there is a point where $\gamma$ is tangent to $X$. But this means that $\alpha=0$. On the other hand, if $\gamma$ crosses a band of $Y$, then there is a point where $\gamma$ is tangent to $Y$. This means that $b(t_0)=0$ for some $t_0 \in I$. By Observation \ref{Observation: b(t_0)=0 implies b=0}, this yields $b$ identically zero. In both cases, this contradicts our assumptions on $\gamma$. 

\textbf{(2)} Indeed, a turnaround, when $\alpha \neq 0$ implies that $b$ vanishes.  
\end{proof}
\begin{corollary}\label{Corollary: incomplete}
Let $l$ be the (boundary) leaf of $X^*$ such that $x(l)=\frac{\pi}{2} + 2 k \pi, k \in \Z$. Let $\gamma$ be a null geodesic of $M$ such that $p:=\Gamma(0)\in l$, $\alpha=1$, and $b(0) >0$.  Then $\gamma$ is incomplete.
\end{corollary}
\begin{proof}
By periodicity, we can fix $l$ to be the (boundary) leaf of $X$ such that $x(l)=\frac{\pi}{2}$. Let $l_1$ be the (boundary) leaf of $Y^*$ such that $x(l_1)=\pi$. In $\R^2$, we can write $\Gamma(t)=(x(t),y(t))$, so that   $\Gamma'(t)=x'(t)\partial_x + y'(t) \partial_y$. On the other hand, $\Gamma'(t)=b(t) X^* + Y^*$ (here we supposed $\alpha=1$). So, a straightforward computation yields 
\begin{align}\label{Eq-loc: x'(t) and y'(t)}
    x'(t)&=b(t)\cos x + \sin x
\end{align}

\noindent\textbf{Step 1:}  We first prove that $\Gamma$ remains in the half band delimited by $l$ and $l_1$, so that $x(t)$ is bounded on $\Gamma$. The geodesic can leave the half band either by crossing the band, or by making a turnaround. 
Looking at the frame field $(X^*,Y^*)$ on the boundary leaves (see Figure 1), it appears that the first behavior cannot occur. More precisely,
if $\Gamma$ crosses the half band, then there is a point $t_0$ such that $x(t_0)=\pi$ and $x'(t_0) \geq 0$. But Equation (\ref{Eq-loc: x'(t) and y'(t)}) yields  $b(t_0)\leq 0$, which is impossible. On the other hand, it follows from  Lemma \ref{Lemma: geodesic behavior} that $\Gamma$ does not make turnarounds. The claim follows.\\
Observe now that in the interior of the half band, we have $f(x,y)=\sin x>0$ and $h(x,y)=\cos x<0$. So for $t$ close enough to $0$ we get
\begin{align*}
    \dot{b}(t) &= \frac{1}{2}f(\Gamma(t)) b^2(t) - h(\Gamma(t)) b(t) \\
            &\geq  \frac{1}{2}f(\Gamma(t)) b^2(t),
\end{align*}
which is then actually true for any $t \in [0,T[$. In particular, $b(t)$ is an increasing function of $t$, and $\lim_{t \to T} b(t)=b_{\infty} \in \R_{>0} \cup \{+\infty\}$. 

\noindent\textbf{Step 2 ($f>0$ and bounded away from $0$ on $\Gamma$):} The way the frame field $(X,Y)$ varies in the half band forces the existence of some $\delta >0$ such that $\Gamma$ remains in the region $\frac{\pi}{2}<x<\pi - \delta$ for $t$ close enough to the limit $T$. Assume this is not the case. Then one can construct a sequence of points $t_n \to T$ such that $x(t_n)$ goes to $\pi$. For such a sequence, we have $\lim_{n}x'(t_n)=-b_{\infty} \in \R_{<0} \cup \{-\infty \}$. If $x'(t)$ was negative for $t$ close enough to the limit, meaning that $x(t)$ is decreasing, one cannot have $x(t_n)$ converging to $\pi$ (since $x(t)<\pi$ for any $t \in [0,T[$). It follows that there is a sequence $t'_n \to T$ such that $x'(t'_n)=0$. But this means that $b(t)$ goes to $0$, and this is a contradiction (remember that $b(t)$ is increasing and $b(0)>0$).  The claim follows. There is no loss of generality in assuming that the claim holds for any $t \in [0,T[$. Then $f(\Gamma(t)) > \epsilon$ for any $t \in [0,T[$, and some $\epsilon>0$. 

\noindent\textbf{Conclusion:} Any solution of the differential equation $\dot{\tilde{b}}(t)=\frac{\epsilon}{2} \tilde{b}^2(t)$, with initial condition $\tilde{b}(0) >0$, is incomplete, defined on a bounded interval $[0, T_0[$, where $T_0 < \infty$. For such a solution $\tilde{b}(t)$, we have $\lim_{t \to T_0} \tilde{b}(t)=\infty$, hence $\lim_{t \to T_0} \dot{\tilde{b}}(t)=\infty$. Since now $\dot{b}(t) \geq \dot{\tilde{b}}(t)$, incompleteness of $\Gamma$ with such an initial value $b(0)$ follows. 
\end{proof}
\begin{comment}
The general completeness question in dimension $3$ in the case of a periodic action reduces to Question $\ref{Question: Dim 3}$. This question deserves further investigation. It seems to us that the completeness depends on the dynamics of the vector fields $X^*$ and $Y^*$ defined on the torus (we hope to come back to it in a 
future work).     
\end{comment}

\subsubsection{\textbf{Case of $\dim G=2$}}\label{subsubsection: dim G=2}
  This case corresponds to a non-periodic null Killing field $V$ with no dense orbit. The fact that there is a $\T^2$-action means that there is another Killing vector field $K$. Since $K$ is preserved by $V$, the norm of $K$ is constant on a $\T^2$-leaf, but it may change from a leaf to another; in particular, the type of $K$ may change. So the leaves of the $\T^2$-action may be either all Lorentzian (when $K$ is timelike for example), all degenerate, or both cases happening on the same manifold. There are special situations where one obtains completeness rather easily, but the general case is largely open. 

\begin{proposition}\label{Proposition: completeness in dimG=2}
Let $(M,V)$ be a compact Lorentzian $3$-manifold with an equicontinuous null Killing field $V$, such that the closure of the flow of $V$ in $\Isom(M)$ is isomorphic to $\T^2$. Assume that the $\T^2$-action is free, and let $\F$ be the $2$-dimensional foliation defined by this action. Then
\begin{itemize}
    \item[(1)] If all the $\F$-leaves are degenerate, then the manifold is complete.
    \item[(2)] If all the $\F$-leaves are Lorentzian, then the manifold is complete. 
\end{itemize}
\end{proposition}
\begin{proof}
(1) The group $\T^2$ provides an abelian subalgebra of Killing fields that generates the totally geodesic lightlike foliation tangent to $V^\perp$. By  \cite[Theorem 3]{globke2016locally}, we conclude that $M$ is a (compact) plane wave and, therefore, complete. 

(2) To prove completeness, we will show that there is a timelike Killing field generated by the action of $\T^2$.  This $\T^2$-action is free and proper, so the quotient map  $\pi: M\to M/\T^2 \simeq \S^1$ is a $\T^2$-principal bundle, whose fibers are given by the leaves of $\F$. Denote by $x$ the coordinate on the base space $\S^1$.
Since the $\T^2$-action is isometric, it preserves the unitary vector field $T$ tangent to the orthogonal distribution $T \mathcal{F}^\perp$. It follows that the $\T^2$-action commutes with the flow of $T$. Denote by $C_0$ (resp. $C_x$) the nullcone field of the Lorentzian metric on $F_0$ (resp. on $F_x$). Let $x \in \S^1$, the (non oriented) angle  $\alpha_x$ of the timecone delimited by $C_x$ is well-defined; indeed, since $\T^2$ is abelian, the Lorentzian metric on the fiber preserved by the $\T^2$-action is flat, so the nullcone field $C_x$ is parallel along the fiber. Moreover, since the parameter $x$ varies in a compact set, and the nullcone field varies continuously, $\alpha_x$ has a minimum $>0$. Now, to guarantee the existence of a timelike Killing field, we proceed as follows. For two leaves $F_0$ and $F_x$, $x \in \S^1$, there is a diffeomorphism $\phi^{t_x}_T: F_0 \to F_x$, where $\phi^t_T$ is the flow of $T$. The pullback of the nullcone field on $F_x$ by the flow of $T$ is another nullcone field on $F_0$, whose (pullbacked) angle is $\alpha_x$. Since $M$ has a null Killing field $V$ commuting with $T$, the nullcone field $C_0$ and the pullbacks of the $C_x$, for all $x \in \S^1$, have a common (invariant) null vector field $V_{\vert F_0}$.
This means that we can choose a timelike vector $\z$ tangent to $F_0$ at some point $p \in F_0$, with the property that all its push-forwards by the $T$-flow are also timelike. Since the (isometric) action of $\T^2$  is transitive on the fibers, there exists a Killing field $K$ such that $K(p)=\z$. Since $K$ is $T$-invariant and all  $d\phi_T^t(\z)$ are timelike, $K$ is timelike along the orbit of $T$ through $x$ (which intersects all the fibers). Finally, the $\T^2$-invariance of $K$ proves that it is timelike everywhere.
\end{proof}

\begin{remark}\label{remark 7.18}
    In the proof of Item $(2)$ of Proposition \ref{Proposition: completeness in dimG=2}, the flow of $T$ does not necessarily preserve the null cone fields of the $\F$-leaves. In fact, the latter are preserved by the flow of $T$ if and only if the flow induces a conformal diffeomorphism between the leaves. The foliation $\mathcal{F}^\perp$ is totally geodesic $($equivalently, $\mathcal{F}$ is transversely Riemannian$)$.  And since $\dim \mathcal{F}^\perp=1$,  $\mathcal{F}^\perp$ is transversely conformal if and and only if the flow of $T$ preserves the conformal structure of the leaves of $\F$. In this case, $\mathcal{F}$ is  umbilic by \cite[Fact 2.1]{Zeghibwarped}.
    When this occurs, by \cite[Proposition 2.4]{Zeghibwarped}, this leads to the warped product structure in Example \ref{Example: dimG=2, warped} below.   
\end{remark}

We mentioned in the proof of Proposition \ref{Proposition: completeness in dimG=2} that the metrics in Item (1) are plane waves.  Let us now look more closely to the situation in Item (2), and see which kind of metrics we can have there. 
\begin{example}[Warped product]\label{Example: dimG=2, warped}
    Consider $\R^3$ with the Lorentzian metric $g_f=f(z)(dx^2 -dy^2)+dz^2$, where $f>0$ and $\tau$-periodic. When $f$ is non-constant, $(\R^3, g_f)$ is non-flat. The connected component of the isometry group of $g_f$ contains $\Sol$, which acts by preserving the Lorentzian flat planes spanned by $\partial_x$ and $\partial_y$. Its maximal normal abelian subgroup acts by translations in the $x$ and $y$ coordinates. And it contains a one parameter subgroup acting by $A^t: (x,y,z) \mapsto (e^t x, e^{-t}y, z)$. Since $f$ is $\tau$-periodic, there is in addition an isometric action of $\Z$ by $(x,y,z) \mapsto (x,y,z+n\tau), n \in \Z$. Clearly, this action commutes with that of $\Sol$.
    Define a quotient $M_f:=\R^3/\Gamma$ by a discrete subgroup $\Gamma \subset \Sol \times \Z$ of the form $\Z^2\times \Z$, where $\Z^2$ is a lattice of the maximal normal abelian subgroup of $\Sol$. Then:\\
    1) The isometry group of $M$ is non-compact. Indeed, there is $\lambda >0$ for which $A^\lambda$ preserves the lattice $\Gamma$, hence induces a non-trivial action on the quotient. This leads to an isometric action of $\Z$ on $M$, coming from $A^\lambda$. \\
    2) By \cite[Lemma 2.1]{frances}, for a generic $f$, the connected component of the isometry group of $g_f$ is exactly $\Sol$.
    In this case, $M$ is a compact Lorentzian manifold with connected component of the isometry group isomorphic to $\T^2$. The isometry group contains the flow of a null Killing field, and one can choose $\Gamma$ in such a way that this flow is dense in $\T^2$, leading to the situation in Proposition \ref{Proposition: completeness in dimG=2}, Item (2). 
\end{example}

\begin{proposition}\label{Prop: warped 3D non-Brinkmann}
    Consider $X=\R^3$ equipped with a warped product metric $g_f=f(z)dxdy+dz^2$, where $f$ is periodic. Then $(X,g_f)$ admits a parallel null vector field $V$, i.e. the warped product is a Brinkmann manifold, if and only if $g_f$ is flat.
\end{proposition}
    \begin{proof}
        First observe that $\Isom^{\mathsf{o}}(X)$ contains $\Sol$, which acts by preserving the flat Lorentzian leaves. Assume $X$ is a Brinkmann manifold. Let $V$ be the parallel null vector field of the  Brinkmann structure. Since the warped product is non-flat, the function $f$ is non-constant. Therefore, it follows from \cite[Chapter 7, Proposition 35]{o1983semi} that the Killing fields generated by one-parameter groups of $\Sol$ cannot be parallel. Hence, the one-parameter group generated by $V$  in $\Isom^{\mathsf{o}}(X)$ is not contained in $\Sol$. So the dimension of $\Isom^{\mathsf{o}}(X)$ is at least $4$. By \cite[Theorem 4.1]{frances}, $X$ must be locally homogeneous. On the other hand, since $f$ is periodic, $X$ admits a compact quotient $M=X/\Gamma$ with a non-compact isometry group (see Example \ref{Example: dimG=2, warped}). This compact quotient is locally homogeneous. By \cite[Theorem C point (2)]{frances}, it is flat, and then $g_f$ is also flat. 
    \end{proof}
As noted in Remark \ref{remark 7.18}, the decomposition coming from an isometric action of $\T^2$ with Lorentzian leaves is a warped product decomposition if and only if $\F^\perp$ is transversely conformal.  In the next example, however, $\F^\perp$ is not transversely conformal. Moreover, the metric is not isometric to a metric $g_f$ as in Example \ref{Example: dimG=2, warped}.

\begin{example}[Non-warped product]\label{Example: dim G=2, non-warped}
Consider $X:=\R^3$ with the Lorentzian metric $g=(dudv + \cos(t) du^2) + dt^2$. 
Let $V:=\partial_v, U:=\partial_u, T:=\partial_t$. Both $U$ and $V$ are Killing fields (with $V$ null) that generate an $\R^2$-isometric action with Lorentzian flat leaves. The null cone  field is spanned by $U$ and $U-\frac{1}{2}cos(t) V$. Since $[T, U-\frac{1}{2}\cos(t) V]=\frac{1}{2}\sin(t) V$, the line field $\R(U-cos(t) V)$ is not preserved by $T$. So the flow of $T$ does not preserve the null cone field, and the decomposition $(\R^2 \times \R, (dudv + \cos(t) du^2) + dt^2)$ is therefore not a warped product with respect to our decomposition. 
Observe that this example is a (non-flat) pp-wave, written in Brinkmann coordinates. So, by Proposition \ref{Prop: warped 3D non-Brinkmann}, it is not isometric to a warped product of the form $g_f$, independently of the decomposition. Moreover, $\R^2$ acts isometrically by translations in the $u$ and $v$ coordinates. Thus, taking the quotient of $X$ by a discrete group of isometries generated by a lattice in $\R^2$ along with the isometry $\gamma:(u,v,t) \mapsto (u,v,t+2\pi)$ results in a compact quotient where the connected component of the isometry group is precisely $\T^2$, with Lorentzian $\F$-leaves. \medskip
\end{example}

The authors do not know at this stage if there are examples in Item (2) where the universal cover is neither a warped product nor a Brinkmann spacetime. A broader question is: what compact Lorentzian $3$-manifolds arise when we assume the existence of an equicontinuous null Killing field?

\addtocontents{toc}{\protect\setcounter{tocdepth}{0}}

\addtocontents{toc}{\protect\setcounter{tocdepth}{1}}

\bibliographystyle{plain}
\bibliography{Bibliography.bib}
\end{document}